\documentclass[12pt,a4paper]{amsart}
\usepackage[dvips]{graphics}
\usepackage{amsmath, amsthm, amssymb, amsfonts}
\usepackage[cp 1250]{inputenc}
\usepackage[T1]{fontenc}
\usepackage[dvips]{graphicx}
\usepackage{amscd}
\usepackage{a4wide}

\newtheorem{thm}{\sc Theorem}[section]
\newtheorem{lemma}[thm]{\sc Lemma}

\newtheorem{rem}[thm]{\sc Remark}
\newtheorem{cor}[thm]{\sc Corollary}
\newtheorem{prop}[thm]{\sc Proposition}

\theoremstyle{definition} 
\newtheorem{defi}[thm]{\sc Definition}

\usepackage[all]{xy}
\newcommand{\tp}{\xymatrix{*+<.7ex>[o][F-]{\scriptstyle\top}}}

\newcommand{\widebar}[1]{\overline{#1\phantom{'}\!}}
\newcommand{\A}{\mathsf{A}}
\newcommand{\C}{\mathbb{C}}
\renewcommand{\i}{{\rm i}}
\newcommand{\I}{{\rm i}}
\newcommand{\id}{{\rm id}}
\newcommand{\Mor}{{\rm Mor}\,}
\newcommand{\one}{1_\A}

\newcommand{\e}{\varepsilon}

\newcommand{\G}{\mathbb{G}}

\newcommand{\N}{\mathbb{N}}

\newcommand{\R}{\mathbb{R}}
\newcommand{\T}{\mathbb{T}}

\newcommand{\co}[2]{c_{#1}^{(#2)}}

\begin{document}

\title[Woronowicz Construction for $N=3$]{Woronowicz construction of compact
quantum groups for functions on permutations. \\ Classification result for 
$N=3$}
 \author{Anna Kula}
 \date{\today}
\address{AK: Instytut Matematyczny, Uniwersytet Wroc{\l}awski,  pl. Grunwaldzki
2/4, 50-384 Wrocław, Poland. {\em On leave from} Insitute of
Mathematics, Jagiellonian University, ul. {\L}ojasiewicza 6, 
Krak{\'o}w, Poland.}
\email{anna.kula@math.uni.wroc.pl}
\begin{abstract}
We provide a classification of compact quantum groups, which can be 
obtained by the Woronowicz construction, when the arrays used in the twisted
determinant condition are extensions of functions on permutations.
General properties of such quantum groups are revealed with the aid of operators
intertwining tensor powers of the fundamental corepresentation. Two new families
of quantum groups appear: three versions of the quantum group $SU_q(3)$ with
complex $q$ and a non-commutative analogue of the semi-direct product
of two-dimensional torus with the alternating group $A_3$. 
\end{abstract}
\maketitle

\section{Introduction}

The Woronowicz construction \cite{woronowicz88}
provides a very general and elegant method of creating compact quantum groups. 
According to this construction, the universal unital C${}^*$-algebra generated
by abstract elements making up a $N\times N$ unitary matrix $u$ 
and satisfying a modified (twisted) determinant
condition admits the quantum group structure. The modification of the
determinant condition depends on a choice of $N^N$ complex constants
$E_{i_1,i_2,\ldots, i_N}$, $i_k\in \{1,2,\ldots,N\}$, forming a 'nondegenerate
array'. The nondegeneracy condition still leaves a lot
of freedom to choose the constants--at least in theory. In
practice, it is extremely difficult to find an example which is non-trivial.
The latter means that the C$^*$-algebra is noncommutative
and thus the quantum group does not originate from any classical group. 

Until present there have been only few examples of quantum groups successfully
constructed with the aid of the Woronowicz result. These are: the twisted
special unitary group
$SU_q(N)$ \cite{woronowicz88}, the $q$-deformed unitary group $U_q(2)$ 
\cite{wysoczanski04} and the $q$-deformed special orthogonal group $SO_q(3)$
\cite{lance98}. Recently, the author showed that also a
twisted product of $SU_{-1}(2)$ with the two-dimensional torus
can be added to this list (see \cite{kula13}). 

On the other hand, the construction is fruitful in the sense that it helped
describing the quantum group structure. For instance, Wysocza\'{n}ski observed
that $U_q(2)$ is a twisted product of $SU_q(2)$ with the unit circle, while
Lance found a C$^*$-embedding of $SO_q(3)$ into $SU_{q^{1/2}}(2)$. 
Both results allowed to provide systematic descriptions of irreducible
representations of the quantum groups in question as well as the related Haar
states.

Several examples of objects coming from Woronowicz construction, which
are known, contrast vividly with the broad range of choice of the constants in
the array $E=[E_{i_1,i_2,\ldots,
i_N}]_{i_1,i_2,\ldots, i_N=1}^N$. It is thus natural to wonder what kind
of arrays $E$ lead--through the Woronowicz construction--to non-trivial quantum
groups. However, different arrays can lead to the same quantum
group (as will become clear from the results of this paper); therefore, it is
better to rephrase the question as follows: what are the non-trivial quantum
groups which can be obtained from the Woronowicz construction?  

One can observe that all the examples, except $SO_q(3)$, are related
to arrays which are trivial extensions of nowhere-zero \emph{functions on
permutations}. By this we mean that $E_{i_1,i_2,\ldots, i_N}$ is non-zero if and
only if $(i_1,i_2,\ldots, i_N)$ is a permutation of the set $\{1,2,\ldots,N\}$.
For instance, for $SU_q(N)$ the non-zero constants are
given by $E_{\sigma}=(-q)^{\mbox{inv} (\sigma)}$, where inv
denotes the function counting the number of inversion in the permutation. In
case of the quantum group $U_q(2)$, the function used in
the construction associates a permutation $\sigma \in S_3$ with the value
$(-q)^{3-\mbox{c} (\sigma)}$, where $\mbox{c} (\sigma)$ stands for the
number of cycles in $\sigma$. In the last known case of
$SU_{-1}(2)\ {\rule{0.4pt}{1.2ex}\! \smash{\times}}_\alpha \ \T^2$, a
transposition-coloring function was used (see \cite{kula13} for details). 

The aim of this paper is to study compact quantum groups coming from the
Woronowicz construction under the assumption that arrays $E$ are trivial
extensions of (nowhere-zero) functions on permutations. In such case, the
nondegeneracy condition is always satisfied. Moreover, we can use the idea
included in the proof of the Woronowicz construction to derive modular
properties of the quantum group  as well as commutation relations between its
generators. In particular, we show (Corollary \ref{cor_modular_prop}) that such
a quantum group is always a quantum subgroup of the van Daele and Wang's
universal unitary quantum group $A_u(M)$, where $M$ is a diagonal matrix with
coefficients depending on $E$. Morphisms which intertwine tensor powers of the
fundamental corepresentation are the key ingredients in these calculations. This
method replaces routine but tedious calculations (usually omitted as in
\cite{bragiel89} or \cite{lance98}) and, to the best of our knowledge, it has
not yet been used in this context. 

The main result of this paper (Theorem \ref{thm_classification}) asserts that
the Woronowicz construction for $N=3$, with $E$ as above, allows only a
restricted type of compact quantum groups. Some of them are known: the quantum
version of the two-dimensional torus, the quantum $U_q(2)$ with complex $q$
\cite{zhang_zhao_2005}, the quantum $SU_q(3)$ with $q$ real. Nevertheless, two
new families appear too; these are three versions of $SU_q(3)$ with complex
parameter $q$ and a non-commutative analogue of the semi-direct product of
two-dimensional torus $\T^2$ with the alternating group $A_3$, denoted by
$A_{p,k,m}(3)$. We will focus primarily on the question of non-triviality of
the new quantum groups, postponing deeper studies of their structure to a
separate paper.

The classification result may be interpreted as reflecting a very specific way
in which the construction was carried out. In fact, the twisted determinant
condition comes from the property that the third power of the fundamental
corepresentation intertwines the identity and, in this sense, it mimics the
property of the classical special unitary group. But the classification result
reveals also a great intuition of S. L. Woronowicz who was able to find a rare
set of parameters leading to the non-trivial quantum group $SU_q(3)$.  

Let us mention that the analogous classification problem for $N=2$ is
already solved. Indeed, using Theorem A1.1 from \cite{woronowicz87a}, one can
deduce that if an array $E=[E_{ij}]_{i,j=1,2}$ is a
trivial extension of a function on permutations, then the only quantum groups
which we get by the Woronowicz construction are the twisted $SU_q(2)$ (if
$\frac{E_{21}}{E_{12}}=-q$ is real) and the quantum version of the unit circle
$\T$ (otherwise). 

\medskip
The paper is organized as follows. In Section \ref{sec_prelim} we recall the
Woronowicz construction and gather necessary definitions. Modular properties
as well as commutation relations between generators are deduced from
morphism properties in Sections \ref{sec_morphisms} and \ref{sec_Q}. Additional
relations, which do not refer explicitly to morphism property, are presented
in Section \ref{sec_additional}. In Section \ref{sec_definitions} we define 
the new quantum groups arising from the construction and answer the question of
their non-triviality. Section \ref{sec_decomposition} is devoted to study the
case when the fundamental corepresentation decomposes into two blocks. Finally,
the proof of the classification theorem represents the content of Section
\ref{sec_class}. 

Throughout the paper, we assume that $N=3$, although the general theory
recalled in Section \ref{sec_prelim} can be stated for arbitrary dimension. We
use the notation $\C^*=\C\setminus\{0\}$, $\R^*=\R\setminus\{0\}$. A scalar
product will always be linear
with respect to the second variable.

\section{General theory}
\label{sec_prelim}

We recall in this Section the construction from \cite{woronowicz88} and
some notation to be used in the sequel. We refer to \cite{woronowicz87b} and
\cite{timmerman} for more detailed information about compact quantum groups. 

There exist various (not always equivalent) definitions of compact
quantum groups. The one which we adopt here comes from \cite{woronowicz87b},
see also \cite{vandaele_wang96}, and is the most suitable in our context. By a
\emph{compact (matrix) quantum group} we understand a pair $\G=(\A,u)$, where
$\A$ is a unital C$^*$-algebra and $u$ is a $N\times N$ matrix with coefficients
in $A$ such that:
\begin{enumerate}
 \item the coefficients $u_{ij}$, $i,j=1,\ldots, n$, generate $\A$,
 \item there exists a C$^*$-homomorphism $\Delta: \A \to \A \otimes \A$ such
that $\Delta(u_{ij})=\sum_{k=1}^n u_{ik}\otimes u_{kj}$,
 \item the matrices $u=[u_{ij}]$ and $u^t=[u_{ji}]$ are invertible. 
\end{enumerate}
The $*$-subalgebra generated by $u_{ij}$'s admits the Hopf structure, with the
counit $\e(u_{ij})=\delta_{ij}$ and the antipode $\kappa(u_{ij})=
(u^{-1})_{ij}$.

\medskip
For a fixed $N\in \N$ let us consider a $N^N$-array of complex numbers
$E=[E_{i_1,i_2,\ldots,i_N}]_{i_1,i_2,\ldots,i_N=1}^N$. We say that the array $E$
is {\em left nondegenerate} if the arrays $E_{1-}, \ldots,E_{N-}$ are linearly
independent, where $E_{k-}=[E_{k,i_2,\ldots,i_N}]_{i_2,\ldots,i_N=1}^N$.
Similarly, the array $E$ is said to be {\em right nondegenerate} if the arrays
$E_{-1}, \ldots,E_{-N}$ are linearly independent, where
$E_{-k}=[E_{i_1,\ldots,i_{N-1},k}]_{i_1,\ldots,i_{N-1}=1}^N$.

\begin{thm}[Woronowicz Construction, Theorem 1.4 in \cite{woronowicz88}] 
Let $E$ be a left and right nondegenerate $N^N$-array of complex
numbers, and let $\A$ be the universal C${}^*$-algebra with unit $1_\A$
generated by matrix
coefficients $u=[u_{jk}]_{j,k=1}^N$ satisfying
\begin{itemize}
\item[a)] {\bf unitarity condition}:
\begin{equation}\tag{U}\label{U}
\sum_{s=1}^{N} u_{js}u_{ks}^* = \delta_{jk}1_\A= \sum_{s=1}^{N}
u_{sj}^*u_{sk}
\end{equation}
\item[b)] {\bf twisted determinant condition}:
\begin{equation}\tag{TD}\label{TD}
\sum_{i_1,i_2,\ldots,i_N=1}^{N} E_{i_1,i_2,\ldots,i_N}
u_{a_1i_1}u_{a_2i_2}\ldots u_{a_Ni_N} = E_{a_1,a_2,\ldots,a_N}\cdot 1_\A.
\end{equation}
\end{itemize}
Then the pair $\G=(\A,u)$ is a compact quantum group. 
\end{thm}
We shall call $\G$ as above a \emph{compact quantum group related to $E$, coming
from Woronowicz construction}. The quantum group structure on $\G$ is then
imposed by the fact that $u$ is a corepresentation, so the usual formulas for
comultiplication, counit and antipode hold. Namely, 
$$ \Delta(u_{jk}) =\sum_{s=1}^N u_{js}\otimes u_{sk}, \quad
\e(u_{jk})=\delta_{jk}, \quad \kappa (u_{jk})=u_{jk}^*.$$
The matrix $u$ is called the \emph{fundamental corepresentation}. 

Given a compact quantum group $\G$ there exists the Haar state $h \in
\mathsf{A}'$ and it is a KMS-state, i.e.  $h(ab) =h(b\sigma_{-i}(a))$ for any $
a,b\in \A$, where $\sigma_{-i}$ is an analytic extension of the modular
automorphisms group $(\sigma_t)_{t\in \R}$ (see for example \cite[Definition
5.3.1]{bratteli+robinson97}). The KMS-property encodes the modular properties
of $\G$, which describe how far from being tracial the Haar state is. If $h$
happens to be tracial, then we say that $\G$ is {\em of Kac type}.

It is well known (cf.\ \cite{woronowicz87b}) that if the C$^*$-algebra $\A$ of
a compact quantum group is commutative, then there exists a (classical) compact
group $G\subset \mathbb{GL}(N)$ such that $\A=C(G)$ is the algebra of all
continuous functions on $G$, the generator $u_{ij}$ associates a matrix $g\in
G$ with the $(i,j)$-th coefficient of $g$, and the comultiplication of functions
from $\A$ reflects the group multiplication: $\Delta(f) (g,g')=f(gg')$. Then
$\G=(\A,[u_{ij}]_{i,j=1}^N)$ is usually called a \emph{quantum version} of the
group $G$. Quantum versions of classical groups represent a very
well-understood class of compact quantum groups (in particular, they are
always of Kac type). That is why, in this paper, we shall refer to them as to
\emph{trivial} compact quantum groups. 

\medskip
We say that a $N^N$-array $E$ satisfies the \emph{permutation condition} if
\begin{equation} \tag{P} \label{perm_cond}
 E_{i_1,i_2,\ldots,i_N}\neq 0 \quad \mbox{if and only if}\quad
(i_1,i_2,\ldots,i_N)\in S_N.
\end{equation}
This means that there exists a complex, nowhere-zero function $f$ on $S_N$, the
set of permutations of $N$ elements, such that $E$ is the trivial extension of 
$f$, that is
$$E_{i_1,i_2,\ldots,i_N}=\left\{ \begin{array}{ll} f(\sigma) & \mbox{if }
(i_1,i_2,\ldots,i_N)=\sigma\in S_N\\
0 & \mbox{if }  (i_1,i_2,\ldots,i_N)\not\in S_N\end{array}\right. .$$
It is easily seen that the permutation condition \eqref{perm_cond} ensures
that
the array $E$ is left and right nondegenerate and thus the Woronowicz
construction applies. 

It will be convenient to use the following (technical) definition.
\begin{defi}
 A \emph{PW-quantum group of dimension $N$} is a compact quantum group $\G$ for
which there exists a $N$-dimensional array $E$ satisfying the permutation
condition \eqref{perm_cond} and such that $\G$ comes from the Woronowicz
construction related to the array $E$.
\end{defi}

Let $E$ be an array related to a PW-quantum group (in particular, $E$ must
satisfy the permutation condition). Then we can multiply all constants by
a non-zero scalar and the relations remain unchanged. Therefore, without
lose of generality, we can assume the normalization $E_{12\ldots N}=1$.

The main problem investigated in this paper can now be shortly rephrased:
classify all PW-quantum groups of dimension 3. The answer which we aim to show
is the following. 

\begin{thm}[Classification Theorem] \label{thm_classification}
The following is a complete list of PW-quantum groups of dimension 3:
 \begin{enumerate}
  \item $C(\T^2)$, quantum version of two-dimensional torus,
  \item $U_q(2)$ for $q\in \C^*$, see Definition \ref{def_uq2},
  \item $A_{p,k,m}(3)$ for $p\in \C^*$ and $k,m\in \{0,1,2\}$, see Definition
\ref{def_akm},
  \item $SU_{p,m}(3)$ for $p\in \C^*$ and $m \in \{0,1,2\}$, see
Definition \ref{def_su_k}.
 \end{enumerate}
\end{thm}

The proof of Theorem will be given in Section \ref{sec_class}. We emphasize,
however, that the meaning of this result is that, although a lot of freedom
is given in the Woronowicz construction, only few quantum groups can appear.
Moreover, it will become clear from Section \ref{sec_definitions} that the
non-trivial ones are $U_q(2)$ with $q\neq 1$, $SU_{p,m}(3)$ with $p\neq 1$ and
$m\neq 0$, and $A_{p,k,m}(3)$ with $k\neq m$ and $k+m \not \in 3\mathbb{N}$. 

\begin{rem} \label{rem_notation}
For $N=3$, an array $E$ satisfying the permutation condition (P) and the
normalization $E_{123}=1$ is uniquely defined by 5 complex, non-zero parameters.
In the sequel, we shall frequently use these parameters, adopting the following
notation:
\begin{eqnarray*} 
&E_{123}=1, \; E_{132}=p, \; E_{213}=q, & \\ 
&E_{231}=r, \; E_{312}=s, \; E_{321}=t. &
\end{eqnarray*}
\end{rem}

\begin{rem} \label{rem_convention}
The indices used to number the generators in the array $u=[u_{ij}]_{i,j=1}^3$
always change in $\{1,2,3\}$. We try to use $r,i,k$ to number the rows in the
array and $n,j,l$ to number the columns. This is not always possible due to
the antipode property, see Corollary \ref{cor_modular_prop}. Moreover, starting
from Section \ref{ssec_Q}, whenever a triple $(r,i,k)$ or $(n,j,l)$ appear, we
tacitly assume that this is an element of $S_3$. For example, when $r$ and $i$
are given and $r\neq i$, then the index $k$ is the unique element different from
the other two, that is such that $(r,i,k)\in S_3$. 
\end{rem}

\section{Morphisms of PW-quantum groups}
\label{sec_morphisms}
Throughout this Section, whenever $\G$ appears, it is understood to be a
PW-quantum group related to a normalized array $E$ of dimension 3. We shall
deduce here several relations which hold for generators of such $\G$. For that
purpose we recall the main idea of the proof of the Woronowicz construction,
which establishes morphisms between tensor powers of the fundamental
corepresentation. Then we show how the morphism properties translate into
relations. 

\subsection{Idea of the Woronowicz construction} 

Each compact quantum group $\G=(\A, u)$ gives rise to the so-called {complete
concrete monoidal W$^*$-category} $R$, cf. \cite{woronowicz88}. 
In this category, objects ${\rm Obj}(R)$ are finite-dimensional unitary
corepresentations $u_s$ on (some) Hilbert spaces $H_s$ with a distinguish
element $u$--the fundamental corepresentation. Morphisms of this category are
intertwining operators
$\Mor (u_s,u_t) :=\{T\in B(H_s,H_t) : u_s(T\otimes \id) =
(T\otimes \id) u_t\}$, where $\id$ denotes the identity on $\A$.  

In \cite[Theorem 1.3]{woronowicz88}, Woronowicz showed that the converse is
true. Namely, if $R$ is a complete concrete monoidal W$^*$-category with a
distinguished object $u$, and if the conjugate object $\bar{u}$ exists and $\{u,
\bar{u}\}$ generates $\text{Obj} (R)$, then the universal C$^*$-algebra $\A$
generated by matrix coefficients of $u$, together with $u$ as fundamental
corepresentation is a compact quantum group. 

This result lies in the core of the proof of the Woronowicz construction.
Namely, given the C$^*$-algebra $\A$ with $u$ satisfying (U) and (TD),
it is possible to construct a complete concrete monoidal W$^*$-category
$R(\G)$, which leads to $\G=(\A,u)$. Objects of this category are generated by
tensor powers of the matrix $u$
$$ {\rm Obj}_0 (\G):=\{ u^{\tp n}: n\in \N\}.$$
Morphisms are linear
combinations of elements which are composition of
mappings of the form $I_k\otimes E \otimes I_l$ and  $I_k\otimes E^* \otimes
I_l$, where $E$ is a fixed operator. 
The operator $E$ serves as the main building block for morphisms between
different elements in $ {\rm Obj} (\G)$. 

\subsection{Morphisms} 

Let $e_1,e_2,e_3$ be the standard orthonormal basis in $\C^3$. The
fundamental corepresentation $u$ belongs to $M_3(\A) \cong B(\C^3) \otimes \A$
and thus can be expressed as
$$ u= \sum_{k,l} m_{kl} \otimes u_{kl},$$
where $m_{kl}$ are standard matrix units (i.e. $m_{kl}e_j=\delta_{jk}e_k$).
Given two elements $v\in M_{n'}(\A)$, $w\in M_{n''}(\A)$, for some $n',n''\in
\mathbb{N}$, 
$$ v = \sum_{j,k=1}^{n'} m'_{jk} \otimes v_{jk} w_{st}, \quad 
w = \sum_{s,t=1}^{n''} m''_{st} \otimes v_{jk} w_{st},$$
where $m'_{jk}$ and $m''_{st}$ are matrix units in
$M_{n'}(\A)$ and $M_{n''}(\A)$ respectively, 
we define the \emph{tensor product of corepresentations} $v$ and $w$ 
$$ v \tp w = \sum_{j,k=1}^{n'} \sum_{s,t=1}^{n''} m'_{jk}
\otimes m''_{st} \otimes v_{jk} w_{st}.$$

Let $H_n=(\C^3)^{\otimes n}$ and let $I_n$ denotes the identity operator on
$H_n$. We define the following mappings:
\begin{eqnarray*}
 \label{eq_E}
 E: \C \ni 1 &\mapsto& \sum_{i,j,k} E_{ijk} e_i \otimes e_j
\otimes e_k \in H_3, \\
 \label{eq_P}
 P=(E^*\otimes I_1)(I_1 \otimes E) : H_1 \ni e_x &\mapsto& \sum_a
\left(\sum_{j,k} \bar{E}_{xjk} E_{jka}\right) e_a \in H_1, \\
 \label{eq_Q}
 Q=(E^*\otimes I_2)(I_2 \otimes E) : H_2 \ni e_x\otimes e_y
&\mapsto&
\sum_{a,b} \left( \sum_k \bar{E}_{xyk} E_{kab} \right) e_a \otimes e_b \in H_2. 
\end{eqnarray*}
The explicit expressions follow from $E^*(e_i \otimes e_j \otimes
e_k)=\bar{E}_{ijk}$. Moreover, we observe that the space $\bar{H} :=Q(H_2)$ is
spanned by the vectors 
$$ x_k=\sum_{a,b} E_{kab} e_a \otimes e_b, \quad k=1,2,3$$
(cf. formula (4.2) in \cite{woronowicz88}). Under the permutation assumption
for $E$, the vectors $x_1,x_2,x_3$ form an orthogonal basis (not orthonormal
in general). 

Furthermore, let $R$ denotes the embedding of $\bar{H}$ into $H_2$. Define
morphisms
\begin{eqnarray*}
T:=(I_1\otimes R^*)E: \C \ni 1 &\mapsto& \sum_{i} e_i \otimes x_i \in 
H_1\otimes \bar{H} , \\
\bar{T}: \bar{H} \otimes H_1 \ni x_j \otimes e_k &\mapsto& \delta_{jk} \in \C.
\end{eqnarray*}
and let
$$ s = \sum_{s,t=1}^3 \bar{m}_{st} \otimes u^*_{st}\in M_3(A)\cong B(\bar{H})
\otimes \A,$$
where $\bar{m}_{jk}$ ($j,k=1,2,3$) are the matrix units in $B(\bar{H})$ with
respect to the orthogonal basis $x_1,x_2,x_3$, i.e. $\bar{m}_{jk} x_n
=\delta_{kn} x_j$. 

It follows from (U) and (TP) (see the
proof of Theorem 1.4 in \cite{woronowicz88}) that 
\begin{enumerate}
 \item[(a)] $s$ is the corepresentation conjugate to $u$ (denoted in the sequel
by $\bar{u}$),
 \item[(b)] $E\in \Mor (1, u^{\tp 3})$, 
 \item[(c)] $P\in \Mor(u,u)$, $Q\in \Mor (u^{\tp 2}, u^{\tp 2})$, $R\in \Mor
(\bar{u}, u^{\tp 2})$, 
 \item[(d)] $T\in \Mor (1, u{\tp}\bar{u})$, $\bar{T} \in \Mor (\bar{u}{\tp}u,
1)$. 
\end{enumerate}
Note that (b) encodes the twisted determinant condition (TP) whereas the
condition (U) is hidden in (d). 

In the following we shall translate these morphism properties into 
relations between algebra generators $u_{jk}$'s. 
\subsection{Decomposition of $u$}

For $\G$ related to the array $E$ we define the \emph{diagonal constants} 
\begin{equation} \label{eq_const_diagonal}
 p_j =\sum_{a,b} \bar{E}_{jab} E_{abj} \quad \mbox{for}\quad  j=1,2,3. 
\end{equation}
\begin{lemma} 
If $p_j \neq p_k$ then $u_{kj}=u_{jk}=0$. 
\end{lemma}

\begin{proof} 
The operator $P=(E^*\otimes I_1)(I_1 \otimes E)$ is by definition a morphism
between $u$ and itself, and (as mentioned above) $P e_j = \sum_k\left(\sum_{a,b}
\bar{E}_{jab} E_{abk}\right) e_k$. But since $E_{jln}\neq 0$ iff $(j,l,n)\in
S_3$, the only non-zero coefficient on the right-hand side will be the one
standing by
$e_j$. Thus $P$ is diagonal, $P e_j = p_j e_j$. 
The fact that $P\in \Mor(u,u)$, i.e.\ $u(P\otimes \id)= (P\otimes \id) u$
means that the two terms
\begin{eqnarray*}
 u(P\otimes \id)&=&
 \big( \sum_{k,l} m_{kl} \otimes u_{kl}\big) 
 \big( \sum_{j} m_{jj} \otimes p_j\one\big)
 = \sum_{j,k} m_{kj} \otimes p_ju_{kj}\\
 (P\otimes \id)u&=&
 \big( \sum_{j} m_{jj} \otimes p_j\one\big)
 \big( \sum_{k,l} m_{kl} \otimes u_{kl}\big) 
 = \sum_{j,l} m_{jl} \otimes p_ju_{jl}
\end{eqnarray*}
are equal (above, we use $m_{kl}m_{ij}=\delta_{il} m_{kj}$). Comparing the
corresponding terms, we see that $p_ju_{kj} = p_ku_{kj}$, from which the
conclusion follows.
\end{proof}

\begin{thm}[Decomposition of the fundamental representation] \label{thm_decomp}
 Let $\G$ be a PW-quantum group related to $E$, with diagonal constants
$p_1,p_2,p_3$. 
Then unless $p_1=p_2=p_3$, $u$ is reducible. 
More precisely, if the diagonal constants are pairwisely different, then
$u=u_{11}\oplus u_{22}\oplus u_{33}$. If $p_i=p_k\neq p_r$ then $u$ decomposes
into two blocks of dimension 1 and 2.
\end{thm}

Note that the fact that $p_1=p_2=p_3$ does not imply that $u$ is irreducible. 
\subsection{Modular properties}
Let us define the \emph{modular constants} related to the array $E$ by
\begin{equation} \label{eq_const_modular}
M_n = \sum_{jk} |E_{njk}|^2 \quad \mbox{for} \quad n=1,2,3.  
\end{equation}
These constants encode the modular properties of the related PW-quantum group in
the sense which will be clear from Corollary \ref{cor_modular_prop}. First, we
show that the transpose of $u$ is similar to a unitary matrix. 

\begin{thm}[Modular properties] \label{thm_modular}
Let $\G$ be a PW-quantum group related to $E$, with modular constants
$M_1,M_2,M_3$. Then the generators of $\G$ satisfy 
\begin{equation} \label{eq_modular_prop} 
\sum_{s=1}^3 M_s u_{sa}u^*_{sb} =\delta_{ab} M_a \one, \quad 
\sum_{i=1}^3 \frac{1}{M_i} u^*_{ki}u_{si}=\frac{1}{M_k}\delta_{ks} \one.
\end{equation}
\end{thm}

\begin{proof}
Recall that $T$ is a linear operator from $\C$ to $H_1\otimes \bar{H}$ defined
by $ T(1)= \sum_{i} e_i \otimes x_i$, and is a morphism between $1$ and
$u{\tp}\bar{u}$. Then from the condition CMW$^*$ III in \cite{woronowicz88} it
follows that $T^*\in \Mor (u{\tp}\bar{u},1)$. It can be checked directly that 
$$T^*(e_a\otimes x_b) = \delta_{ab} M_a, \quad \mbox{where} \quad M_a =\|x_a\|^2
.$$
Thus $(T^*\otimes \id)(e_a \otimes x_b \otimes \one) = \delta_{ab} M_a \one$
while 
\begin{eqnarray*}
 (T^*\otimes \id) (u{\tp}\bar{u})(e_a \otimes x_b \otimes \one) 
 &=& (T^*\otimes \id) \big( \sum_{i,j,s,t} m_{ij}e_a \otimes \bar{m}_{st}x_b
\otimes u_{ij}u^*_{st}\big)
 \\ &=& 
\sum_{i,s} T^*( e_i \otimes x_s) \otimes u_{ia}u^*_{sb}
= \sum_{s} M_s u_{sa}u^*_{sb}.
\end{eqnarray*}
So $(T^*\otimes \id) (u{\tp}\bar{u}) = (T^*\otimes \id)$ yields the first
relation in \eqref{eq_modular_prop}.

To show the second relation, we use the operator $\bar{T}: \bar{H} \otimes H_1
\to \C$ defined by $\bar{T}(x_j \otimes e_k) = \delta_{jk}$. Its adjoint is 
$$\bar{T}^*(1) = \sum_{i} \frac{1}{M_i} x_i \otimes e_i, \quad M_i =\|x_i\|^2.$$
Since we know that $\bar{T} \in \Mor (\bar{u}{\tp}u, 1)$, thus $\bar{T}^* \in
\Mor (1,\bar{u}{\tp}u)$.
Evaluating the morphism condition on $1 \otimes \one$ we have
\begin{eqnarray*}
\lefteqn{  \sum_{i} x_i \otimes e_i \otimes \frac{1}{M_i} \one  = 
(\bar{T}^*\otimes
\id)(1\otimes \one) = (\bar{u}{\tp}u)(\bar{T}^*\otimes \id) (1 \otimes \one)}
\\ &=& 
\big( \sum_{k,j,s,t} \bar{m}_{kj} \otimes m_{st} \otimes u^*_{kj}u_{st}\big)
(\sum_{i} \frac{1}{M_i} x_i \otimes e_i \otimes \one)
= 
\sum_{k,s} x_k \otimes e_s \otimes \sum_{i} \frac{1}{M_i} u^*_{ki}u_{si}.
\end{eqnarray*}
Comparing the coefficients we find out that the second relation in
\eqref{eq_modular_prop} holds. 
To end the proof, it remains to see that $\|x_n\|^2= \sum_{jk} |E_{njk}|^2$.
\end{proof}

Note that the relations \eqref{eq_modular_prop} can be read as 
\begin{equation*} 
 u^tM\bar{u}M^{-1} = I = M\bar{u}M^{-1}u^t,
\end{equation*}
where $M={\rm diag} (M_1,M_2,M_3)$. This fits with the definition
of the universal unitary quantum group, cf. \cite{vandaele_wang96}. The next
result combine
this observation with the standard results from \cite{woronowicz87b}, see also
\cite{timmerman}.

\begin{cor} \label{cor_modular_prop}
Let $\G$ be a PW-quantum group related to $E$, with modular constants
$M_1,M_2,M_3$. Then $\G$ is a quantum subgroup of the universal unitary
quantum group $A_u(M)$, where $M={\rm diag}\ (M_1,M_2,M_3)$. Moreover, 
\begin{enumerate}
 \item the antipode of $\G$ is uniquely defined by 
  \begin{equation} \label{eq_antipode}
  \kappa(u_{kj}) = u_{jk}^* \quad \mbox{ and }
  \quad \kappa(u_{kj}^*) = \frac{M_j}{M_k} u_{jk}; 
  \end{equation}
 \item the Woronowicz characters constants of $\G$ are given by
  $$f_z(u_{ij})=\lambda M_j\, \delta_{ij} \quad \mbox{ with } \quad
  \lambda=\sqrt{(\sum\frac{1}{M_i})(\sum M_i)^{-1}};$$ 
 \item if all $M_j$'s are equal, then $\G$ is of Kac type.
\end{enumerate}
\end{cor}

\subsection{Adjoints of generators}
The morphism $R$ will help us recover the formula for adjoints of the 
generators. Let us recall (cf.\ Remark \ref{rem_convention}) that for $r\in
\{1,2,3\}$ we can find $i,k\in \{1,2,3\}$ such that $i\neq k$ and $(r,i,k)\in
S_3$. 
\begin{thm}
 Given $r,n\in \{1,2,3\}$, the adjoint of the generator $u_{rn}$ in $\G$
satisfies 
\begin{equation} \label{eq_R_star} 
u^*_{rn} = \frac{E_{njl}}{E_{rik}} u_{ij} u_{kl}+ \frac{E_{nlj}}{E_{rik}}
u_{il}u_{kj}
\end{equation}
and 
\begin{equation} \label{eq_R_star_adjoint} 
u^*_{rn} = \frac{\bar{E}_{rik}}{\bar{E}_{njl}} u_{ij} u_{kl}+
\frac{\bar{E}_{rki}}{\bar{E}_{njl}} u_{kj}u_{il},
\end{equation}
where $i,k$ and $j,l$ are chosen to complete $r$ and $n$, respectively, to a
permutation. 
\end{thm}

\begin{proof}
Recall that $R$ is the embedding of $\bar{H}$ into $H_2$, where 
$$\bar{H}=\{\sum_{k=1}^3 \alpha_k x_k; \alpha_k\in \C\}, \quad x_k=\sum_{a,b}
E_{kab} e_a \otimes e_b.$$
This means that 
$$ Rx_k = \sum_{a,b} E_{kab} e_a \otimes e_b \in H_2. $$

Using $ \displaystyle \bar{u} = \sum_{s,t} \bar{m}_{st} \otimes u^*_{st}$,
we compare both sides of $(R\otimes \id)\bar{u}= u^{\tp 2}(R\otimes \id)$ when
evaluated on $x_k\otimes \one$:
\begin{eqnarray*}
 (R\otimes \id)\bar{u}(x_k\otimes \one) 
 &=& 
 \sum_{s,t} R\bar{m}_{st} x_k \otimes u^*_{st}
 = \sum_{a,b} e_a \otimes e_b \otimes \left(\sum_{s} E_{sab} u^*_{sk}\right)
\end{eqnarray*}
and	
\begin{eqnarray*}
\lefteqn{ u^{\tp 2}(R\otimes \id)(x_k\otimes \one) 
 =  
 \big( \sum_{i,j,s,t} m_{ij} \otimes m_{st} \otimes
u_{ij} u_{st}\big) \big(\sum_{a,b} E_{kab} e_a \otimes e_b \otimes \one \big) }
\\ &=& 
\sum_{i,j,s,t} \sum_{a,b} E_{kab} m_{ij}e_a \otimes m_{st}e_b \otimes
u_{ij} u_{st}
 = 
\sum_{i,s} e_i \otimes e_s \otimes \left(\sum_{j,t} E_{kjt} u_{ij} u_{st}
\right).
\end{eqnarray*}
Hence
\begin{equation*} 
\sum_{s} E_{sib} u^*_{sk} = \sum_{j,t} E_{kjt} u_{ij} u_{bt}.  
\end{equation*}
Due to the permutation condition (P), for fixed $i\neq b$ there exists the
unique $s$ such that $(i,b,s)\in S_3$ and $E_{sib}\neq 0$, and exactly two pairs
of $(j,t)$ such that $(k,j,t)\in S_3$, therefore
$$u^*_{sk} = \frac{E_{kjt}}{E_{sib}} u_{ij} u_{bt}+ \frac{E_{ktj}}{E_{sib}}
u_{it}u_{bj}.$$
Changing indices appropriately, we get \eqref{eq_R_star}. 

Performing similar computations for $R^*(e_j \otimes e_l) = \bar{E}_{njl}
x_n$, which is a morphisms between $u^{\tp 2}$ and $\bar{u}$, we end up with the
relation \eqref{eq_R_star_adjoint}.
\end{proof}

\section{Commutation relations for generators}
\label{sec_Q}
We devote this section to exhibit the commutation relations between generators.
For that we use the fact that $Q=(E^*\otimes I_2)(I_2 \otimes E)$ and its
adjoint, acting on $H_2$, intertwine $u^{\tp2}$ with itself. It turns out that
the commutation relations depend on the constants $\co{l}{n}$, defined, for
$l\neq n$, as 
\begin{equation} \label{eq_const_charact}
\co{l}{n} : = \frac{E_{jln}}{E_{jnl}}\frac{E_{nlj}}{E_{lnj}}, 
\end{equation}
where $j$ is the unique integer such that $(j,l,n)\in S_3.$
We shall refer to $\co{l}{n}$'s as to \emph{characteristic constants} of the
array $E$. There are six such constants, but it is immediately seen that
they satisfy
\begin{equation} \label{eq_char_const_rel}
 \co{n}{j}\co{j}{n} = 1, \quad \co{n}{l} \co{l}{j}=\co{n}{j}.
\end{equation}
The main result of this Section is the following.
\begin{thm} \label{thm_comm_rel}
Let $l\neq n$, then the generators in $\G$ satisfy the following commutations:
\begin{enumerate}
 \item[(a)] If $\co{l}{n}\neq 1$ then 
\begin{equation} \label{eq_comm_rel_row0}
u_{an} u_{al}=0 \quad \mbox{and} \quad u_{la} u_{na}=0 \quad
(a=1,2,3) 
\end{equation}
and for $r\neq k$ we have 
\begin{equation} \label{eq_comm_rel_not1}
u_{rn} u_{kl} 
=\frac{E_{irk}}{E_{ikr}} \frac{E_{jln}}{E_{jnl}} \frac{1 -
\co{k}{r}}{1-\co{l}{n}} u_{kl} u_{rn}
+ \frac{E_{irk}}{E_{ikr}} \frac{1 - \co{k}{r}\co{l}{n}}{1-\co{l}{n}} u_{kn}
u_{rl}.
\end{equation}

 \item[(b)] If $\co{l}{n}= 1$ then 
\begin{equation} \label{eq_comm_rel_row1}
u_{an} u_{al} = -\frac{ E_{jln}}{E_{jnl}}\, u_{al} u_{an}, \quad 
u_{na}u_{la} = -\frac{\bar{E}_{jln}}{\bar{E}_{jnl}}u_{la} u_{na}
\end{equation}
and
\begin{eqnarray} \label{eq_comm_rel_1}
u_{rn} u_{kl} 
&=& \frac{E_{irk}}{E_{ikr}}\frac{E_{jln}}{E_{jnl}}\left(1+\co{k}{r} \big|\frac
{E_{ikr}}{E_{irk}} \big|^2\right) \left(1+\big|\frac{E_{jln}}{E_{jnl}}\big|^2
\right)^{-1} \,u_{kl} u_{rn} 
\\ \nonumber & & +\, 
\frac{E_{irk}}{E_{ikr}}
\left(1 -
\co{k}{r}\big|\frac{E_{jln}}{E_{jnl}}\big|^2\big|\frac{E_{ikr}}{E_{irk}}\big|^2
\right)\left(1+\big|\frac{E_{jln}}{E_{jnl}}
\big|^2\right)^{-1} \, u_{kn} u_{rl}. 
\end{eqnarray}
\end{enumerate}
\end{thm}

\begin{cor} \label{cor_comm_rel}
 If $k\neq r$ and $l\neq n$, then there exist complex constants $A_{r,k}^{n,l}$
and
$B_{r,k}^{n,l}$ such that 
$$u_{rn} u_{kl} = A_{r,k}^{n,l} u_{kl} u_{rn} +B_{r,k}^{n,l} u_{kn} u_{rl}.$$
\end{cor}

The rest of this Section is devoted to the proof of Theorem \ref{thm_comm_rel}. 
\subsection{Intertwiner $Q$}
We consider the operator $Q=(E^*\otimes I_2)(I_2 \otimes E):H_2\to H_2$ and
write it as
$$Q=\sum_{a,b,x,y} Q_{ab,xy} m_{ax}\otimes m_{by}, \quad \mbox{where} \quad 
Q_{ab,xy} =\langle e_a\otimes e_b, Q (e_x\otimes e_y) \rangle
= \sum_i E_{iab}\, \bar{E}_{xyi}.$$
Then, with $u=\sum_{l,n}m_{ln} \otimes u_{ln}$, the fact that $Q\in
{\rm Mor}\, (u^{\tp 2}, u^{\tp 2})$, i.e.\  $(Q\otimes I) u^{\tp 2} =
u^{\tp 2} (Q\otimes I)$, means explicitly that
\begin{eqnarray*}
\sum_{l,n,l',n'} [Q \cdot (m_{ln}\otimes m_{l'n'})] \otimes u_{ln} u_{l'n'}&=&
\sum_{l,n,l',n'} [(m_{ln}\otimes m_{l'n'}) \cdot Q] \otimes u_{ln} u_{l'n'}. 
\end{eqnarray*}
Application of $m_{ab}m_{cd}=\delta_{b,c} m_{ad}$ to express
\begin{eqnarray*}
 Q \cdot (m_{ln}\otimes m_{l'n'}) 
 & = & 
 \sum_{a,b,x,y} Q_{ab,xy} (m_{ax}m_{ln}\otimes m_{by}m_{l'n'}) 
 = \sum_{a,b} Q_{ab,ll'}\, m_{an}\otimes m_{bn'}, \\
(m_{ln}\otimes m_{l'n'}) \cdot Q
 & = & \sum_{a,b,x,y} Q_{ab,xy} (m_{ln}m_{ax}\otimes m_{l'n'}m_{by}) 
 = \sum_{x,y} Q_{nn',xy}\, m_{lx}\otimes m_{l'y},
\end{eqnarray*}
leads to
\begin{eqnarray*}
\sum_{a,n,b,n'} m_{an}\otimes m_{bn'} \otimes \sum_{l,l'} Q_{ab,ll'}\,
u_{ln} u_{l'n'}&=&
\sum_{l,x,l',y} m_{lx}\otimes m_{l'y} \otimes \sum_{n,n'} Q_{nn',xy}\, u_{ln}
u_{l'n'} 
\end{eqnarray*}
Comparing the coefficients of the same matrix units we find out that for fixed
$A,B,X,Y$ 
\begin{eqnarray}
\label{eq_Q_main}
\sum_{l,l'} Q_{AB,ll'}\, u_{lX} u_{l'Y}&=&
\sum_{n,n'} Q_{nn',XY}\, u_{An} u_{Bn'} 
\end{eqnarray}
or explicitly (changing indices for convenience) 
\begin{equation} \label{eq_Q_gen}
\sum_i E_{iab} \sum_{r,k}  \bar{E}_{rki} \, u_{rx} u_{ky}
= \sum_j \bar{E}_{xyj} \sum_{n,l} E_{jnl}\, u_{an} u_{bl}.
\end{equation}

\subsection{Relations from $Q$}
\label{ssec_Q}

It follows from the permutation assumption (P) that whenever $a=b$ and $x=y$,
both sides of \eqref{eq_Q_gen} equal zero. On the other hand, if $a\neq b$
then there exists the unique element $i$ such that $(i,a,b)\in S_3$ and
$E_{iab}\neq 0$. Similarly, for $x\neq y$ there is the unique $j$ such that
$(j,x,y)\in S_3$ and $E_{xyj}\neq 0$. Consider the following three cases.
\begin{enumerate}
 \item[{\bf (A)}] 
If $a=b$ and $x\neq y$, then $E_{xyj}\neq 0$ for the unique $j$ described above,
but
$E_{iab}=0$, and \eqref{eq_Q_gen} reduces to 
 $$0=\sum_{n,l} E_{jnl}\, u_{an} u_{al} = E_{jnl}\, u_{an} u_{al}  +
E_{jln}\, u_{al} u_{an}, $$
where in the last part the indices $n$ and $l$ are chosen in such a way that
$(j,l,n)\in S_3$. This implies that two elements in the same row satisfy the
relation 
\begin{equation} \label{eq_Q_row}
 u_{an} u_{al} = -\frac{ E_{jln}}{E_{jnl}}\, u_{al} u_{an}.
\end{equation}

 \item[{\bf (B)}] 
If $a\neq b$ and $x=y$, then $E_{iab}\neq 0$ for the unique $i$
and thus from \eqref{eq_Q_gen} we get the following relation between elements in
one column
\begin{equation} \label{eq_Q_col0}
 u_{rx} u_{kx} =- \frac{\bar{E}_{kri}}{\bar{E}_{rki}} \, u_{kx} u_{rx}. 
\end{equation}

 \item[{\bf (C)}] 
Assume that $a\neq b$ and $x\neq y$, and let $i$ and $j$ denote the unique
indices described above and let $(r,k,i), (j,n,l)\in S_3$. Then \eqref{eq_Q_gen}
leads to
$$\frac{\bar{E}_{rki}}{\bar{E}_{xyj}} \, u_{rx} u_{ky} +
\frac{\bar{E}_{kri}}{\bar{E}_{xyj}} \, u_{kx} u_{ry}
= \frac{E_{jnl}}{E_{iab}}\, u_{an} u_{bl}
+\frac{E_{jln}}{E_{iab}}\, u_{al} u_{bn}.$$
\end{enumerate}
%
%
We see that $\{a,b\}=\{r,k\}$ and $\{x,y\}=\{l,n\}$, so expressing the last
formula for $a=r,b=k,x=l,y=n$ and for $a=k,b=r,x=l,y=n$, and comparing the two
results we find out that
\begin{equation} \label{eq_Q_mix1}
\frac{E_{jnl}}{E_{irk}}\, u_{rn} u_{kl}
+\frac{E_{jln}}{E_{irk}}\, u_{rl} u_{kn}
= \frac{E_{jnl}}{E_{ikr}}\, u_{kn} u_{rl}
+\frac{E_{jln}}{E_{ikr}}\, u_{kl} u_{rn}.
\end{equation}
%
%
On the other hand, taking $a=r,b=k$ and $x=l,y=n$ or $x=n,y=l$ leads to
\begin{equation} \label{eq_Q_mix2}
\frac{\bar{E}_{rki}}{\bar{E}_{lnj}} \, u_{rl} u_{kn} +
\frac{\bar{E}_{kri}}{\bar{E}_{lnj}} \, u_{kl} u_{rn}
= \frac{\bar{E}_{rki}}{\bar{E}_{nlj}} \, u_{rn} u_{kl} +
\frac{\bar{E}_{kri}}{\bar{E}_{nlj}} \, u_{kn} u_{rl}.
\end{equation}

Moreover, if we compare \eqref{eq_Q_mix1} and \eqref{eq_Q_mix2}
we find out that
\begin{equation*} 
\left( \frac{E_{jnl}}{E_{jln}}+\frac{\bar{E}_{lnj}}{\bar{E}_{nlj}} \right)
u_{rn} u_{kl} 
-\left( \frac{E_{irk}}{E_{ikr}}+\frac{\bar{E}_{kri}}{\bar{E}_{rki}}\right)
\,u_{kl} u_{rn} 
=
\left( \frac{E_{jnl}}{E_{jln}}\frac{E_{irk}}{E_{ikr}}
-\frac{\bar{E}_{lnj}}{\bar{E}_{nlj}}  \frac{\bar{E}_{kri}}{\bar{E}_{rki}}
\right)\, u_{kn} u_{rl}
\end{equation*}
%
%
or, in terms of the characteristic constants, 
\begin{equation} \label{eq_Q_mix12}
\begin{array}{c}    \displaystyle
\frac{E_{jnl}}{E_{jln}}\big( 1+ \co{l}{n}
\left|\frac{E_{lnj}}{E_{nlj}}\right|^2 \big) u_{rn} u_{kl} 
-\frac{E_{irk}}{E_{ikr}}\big( 1+ \co{k}{r} 
\left| \frac{E_{kri}}{E_{rki}}\right|^2 \big)\,u_{kl} u_{rn} \\
\displaystyle
=  \frac{E_{jnl}}{E_{jln}}\frac{E_{irk}}{E_{ikr}}
\big(1 - \co{l}{n} \co{k}{r} 
\left| \frac{E_{kri}}{E_{rki}}\right|^2
\left|\frac{E_{lnj}}{E_{nlj}}\right|^2 \big)\, u_{kn} u_{rl}.
\\ \\ \end{array} 
\end{equation}

\subsection{Intertwiner $Q^*$}
\label{rel_3.3}
Consider now the adjoint of $Q$, i.e. $Q^*=(I_2\otimes E^*)(E\otimes I_2) : H_2
\to H_2$. It follows from the condition CMW$^*$ III in \cite{woronowicz88} that
$Q^*\in \Mor (u^{\tp 2},u^{\tp 2})$, so the relation \eqref{eq_Q_main} holds
with $Q_{ab,xy}$ replaced by 
$$Q^*_{ab,xy}=\langle e_a\otimes e_b, Q^* (e_x\otimes e_y) \rangle 
= \sum_k E_{abk}\, \bar{E}_{kxy}. $$
Thus
\begin{equation*}
\sum_k E_{abk}\, \sum_{r,i} \bar{E}_{kri}u_{rx} u_{iy} =
\sum_n \bar{E}_{nxy} \sum_{j,l} E_{jln}\, u_{aj} u_{bl} .
\end{equation*}

\begin{enumerate}
 \item[{\bf (A*)}]  If $a\neq b$ and $x=y$ then 
$ \displaystyle 0 = \sum_{r,i} \bar{E}_{kri}u_{rx} u_{ix} =
\bar{E}_{kri}u_{rx} u_{ix} + \bar{E}_{kir}u_{ix} u_{rx}.$
Hence
\begin{equation} \label{eq_Q_col}
u_{rx} u_{ix} = - \frac{\bar{E}_{kir}}{\bar{E}_{kri}} u_{ix} u_{rx},
\end{equation}
which compared with \eqref{eq_Q_col0}: $ u_{rx} u_{ix} =-
\frac{\bar{E}_{irk}}{\bar{E}_{rik}} \, u_{ix} u_{rx}$ yields, 
for any triple $(r,i,k)\in S_3$, the implication
\begin{equation} \label{eq_options1}
\mbox{if} \qquad  
\co{i}{r}= \frac{\bar{E}_{kir}}{\bar{E}_{kri}}
\frac{\bar{E}_{rik}}{\bar{E}_{irk}} \neq 1 
\qquad \mbox{then} \qquad \forall x\in \{1,2,3\} : \, u_{ix} u_{rx}=u_{rx}
u_{ix}=0. 
\end{equation}

 \item[{\bf (B*)}] Similarly, when $a=b$ and $x\neq y$ we get
$$ u_{aj} u_{al} =- \frac{E_{ljn}}{E_{jln}}\, u_{al} u_{aj},$$
which, in turn, compared with \eqref{eq_Q_row}: 
$u_{aj} u_{al} = -\frac{E_{nlj}}{E_{njl}}\, u_{al} u_{aj},$
leads to 
\begin{equation} \label{eq_options2}
\mbox{if} \qquad  
\co{l}{j}= \frac{E_{nlj}}{E_{njl}}\frac{E_{jln}}{E_{ljn}} \neq 1 
\qquad \mbox{then} \qquad \forall a\in \{1,2,3\} : \, u_{aj} u_{al}=u_{al}
u_{aj}=0. 
\end{equation}

 \item[{\bf (C*)}] Finally, assume that $a\ne b$ and $x\neq y$. Then
(similarly as before, but changing indices to make the comparision easier) 
 $$\sum_{r,k} \frac{\bar{E}_{irk}}{\bar{E}_{nxy}} u_{rx} u_{ky}
 =\sum_{n,l} \frac{E_{nlj}}{E_{abi}}\, u_{an} u_{bl}  $$
implies
$$\frac{E_{nlj}}{E_{rki}}\, u_{rn} u_{kl} 
+ \frac{E_{lnj}}{E_{rki}}\, u_{rl} u_{kn}
= 
\frac{E_{nlj}}{E_{kri}}\, u_{kn} u_{rl}
+\frac{E_{lnj}}{E_{kri}}\, u_{kl} u_{rn}.$$ 
\end{enumerate}
%
The relation compared with \eqref{eq_Q_mix1} gives
$$\left( \frac{E_{jnl}}{E_{jln}}\,
- \frac{E_{nlj}}{E_{lnj}} \right) u_{rn} u_{kl} 
-\left(\frac{E_{irk}}{E_{ikr}} - \frac{E_{rki}}{E_{kri}}\right) u_{kl} u_{rn}
= \left( \frac{E_{jnl}}{E_{jln}}\, \frac{E_{irk}}{E_{ikr}}
- \frac{E_{nlj}}{E_{lnj}}\frac{E_{rki}}{E_{kri}} \right) u_{kn} u_{rl}.$$
or, equivalently,
\begin{equation} \label{eq_QA_mix2}
\frac{E_{jnl}}{E_{jln}}\left( 1 - \co{l}{n} \right) u_{rn} u_{kl} 
-\frac{E_{irk}}{E_{ikr}}\left(1 - \co{k}{r} \right) u_{kl} u_{rn}
= \frac{E_{jnl}}{E_{jln}}\, \frac{E_{irk}}{E_{ikr}} \left( 1
- \co{l}{n}\co{k}{r} \right) u_{kn} u_{rl}.
\end{equation}

\begin{proof}[Proof of Theorem \ref{thm_comm_rel}]
When $\co{l}{n}\neq 1$, the relations \eqref{eq_options1} and
\eqref{eq_options2} imply \eqref{eq_comm_rel_row0}, and \eqref{eq_QA_mix2} gives
\eqref{eq_comm_rel_not1}.  On the other hand, if $\co{l}{n}=1$,
then \eqref{eq_Q_col} and \eqref{eq_Q_row} yield \eqref{eq_comm_rel_row1}. When
$\co{l}{n}=1$, the relation \eqref{eq_QA_mix2} becomes trivial, but
\eqref{eq_Q_mix12} allows to get \eqref{eq_comm_rel_1}. 
\end{proof}

\section{Additional results}
\label{sec_additional}
In this Section we gather three additional results of different types, which we
prove without referring directly to morphisms properties (yet we do use the
commutation relations proved in the previous sections). 

\subsection{Isomorphism lemma}
We remarked in Section \ref{sec_prelim} that, without loss of generality, when
studying PW-quantum groups we can assume the array $E$ to be normalized to
$E_{123}=1$. The lemma which follows shows that much more freedom is
at hand; we can also permute the indices in the array $E$ without affecting
the PW-quantum group.

\begin{lemma} \label{lem_iso}
 Let $\G=(\A, u)$ and $\tilde{\G}=(\tilde{\A}, \tilde{u})$ be two PW-quantum
groups related to the $N$-dimensional arrays $E$ and $\tilde{E}$, respectively.
If there exists a permutation $\sigma\in S_N$ and a constant $c\neq
0$ such that  
$$E_{\sigma(i_1), \ldots, \sigma(i_N)}=c\tilde{E}_{i_1\ldots i_N} \quad \mbox{
for
all } \quad (i_1\ldots i_N)\in S_N,$$
then the groups $\G$ and $\tilde{\G}$ are isomorphic. 
\end{lemma}

\begin{proof}
Let $u_{kn}$, $1\leq k,n \leq N$, are the generators of $\G$ and
$\tilde{u}_{kn}$ the generators of $\tilde{\G}$. Consider the transformation 
$$\Phi: \tilde{\mathbb{G}} \ni \tilde{u}_{kn} \mapsto u_{\sigma(k),\sigma(n)}
\in
\mathbb{G}. $$
We will show that it is well defined and thus extends to a (unital)
$*$-homomorphism on ${\rm Pol}(\tilde{\G})$. But, if this is the case, then
$\Phi$ (being a bijection between the generators of universal algebras) will be
an isomorphism between $\tilde{\mathbb{G}}$ and $\G$. Obviously, it will also
preserve the Hopf-algebra structure 
$$ \Delta \circ \Phi = (\Phi \otimes \Phi) \circ \tilde{\Delta}, \quad 
\e \circ \Phi=\tilde{\e}
\quad \mbox{and} \quad \kappa \circ \Phi = \Phi \circ \tilde{\kappa}.$$ 
 
To prove that $\Phi$ is well-defined, we need to show that the relations which
exist between generators in $\tilde{\mathbb{G}}$ agree (when transformed by
$\Phi$) with those 
between generators in $\G$. For that it is enough to
look on the defining relation from the Woronowicz theorem, namely, the unitarity
\eqref{U} and the twisted determinant condition \eqref{TD}. In the first case we
have
\begin{eqnarray*}
\delta_{kn}\tilde{1}= \sum_{s}^{} \tilde{u}_{sk}^*\tilde{u}_{sn}
&\Rightarrow &
\delta_{kn} \Phi(\tilde{1}) 
=\sum_{s} \Phi(\tilde{u}_{ks})\Phi(\tilde{u}_{ns})^* = \sum_{v=\sigma(s)}
u_{\sigma(k),v}u_{\sigma(n),v}^* 
\\ &\Rightarrow &
\delta_{kn}1= \delta_{\sigma^{-1}(k), \sigma^{-1}(n)}1= \sum_{v}^{} u_{kv}
u_{nv}^* 
\end{eqnarray*}
and we recover the unitarity relation for $\G$. Similarly, $\Phi$ acting
on  $\displaystyle \delta_{kn}\tilde{1}= \sum_{s}^{}
\tilde{u}_{sk}^*\tilde{u}_{sn}$ gives $\delta_{kn}1= \sum_{v}^{} u_{vk}^*
u_{vn}$ which agrees with the definition of $\G$. 

On the other hand, applying $c\cdot \Phi$ to (TD) in $\tilde{\mathbb{G}}$
$$\sum_{i_1\ldots i_N } \tilde{E}_{i_1\ldots i_N} \tilde{u}_{a_1i_1}\ldots
\tilde{u}_{a_1i_1} = \tilde{E}_{a_1\ldots a_N}\cdot \tilde{1}$$
we get 
$$\sum_{i_1\ldots i_N}
u_{\sigma(a_1),\sigma(i_1)}\ldots
u_{\sigma(a_N),\sigma(i_N)}E_{\sigma(i_1),\ldots,\sigma(i_N)} =
E_{\sigma(a_1),\ldots,\sigma(a_N)}\cdot 1,$$
which is the twisted determinant condition for $\G$. 
\end{proof}
\subsection{Partial isometries}
The next result shows that in special cases the  generators $u_{rn}$ must 
be normal partial isometries. In view of Theorem \ref{thm_comm_rel}, this can
happen when (some) $\co{l}{n}$'s  differ from 1. 

\begin{thm} \label{thm_pi}
Let $\G=(\A, u)$ be a PW-compact quantum group. Let $R$ and $J$ be fixed indices
$(R,J\in \{1,2,3\})$ such that  
\begin{equation} \label{eq_part_isom_basic}
u_{Rj}u_{ij} =0, \quad u_{rJ}u_{rl} =0 \quad \mbox{whenever} \quad i\neq R,\; 
l\neq J, \; r,j=1,2,3.
\end{equation}
Then the element $u_{RJ}$ is a normal partial isometry. 
\end{thm}

\begin{proof}
Given $R$ and $J$ we fix the indices $i,k$ and $n,l$ such that $(R,k,i),
(J,n,l)\in S_3$. When $k\neq R$, then using \eqref{eq_R_star} we show that 
\begin{eqnarray*}
 u_{kJ}u_{RJ}^* &=& \frac{E_{Jnl}}{E_{Rki}} \underbrace{u_{kJ}u_{kn}} u_{il}+
\frac{E_{Jln}}{E_{Rki}} \underbrace{u_{kJ} u_{kl}}u_{in}
\stackrel{\eqref{eq_part_isom_basic}}{=}0.
\end{eqnarray*}
The second relation of (U) for $j=k=J$ multiplied by
$u_{RJ}^*$ from the right yields
\begin{eqnarray*}
 u_{RJ}^*= u_{RJ}^*u_{RJ}u_{RJ}^* + u_{kJ}^*(u_{kJ}u_{RJ}^*) +
u_{iJ}^*(u_{iJ}u_{RJ}^*) = u_{RJ}^*u_{RJ}u_{RJ}^* ,
\end{eqnarray*}
so $u_{RJ}$ is a partial isometry. 

Now let us compare the relations (U) with the first relation in
\eqref{eq_modular_prop}:
\begin{eqnarray*}
\displaystyle  u_{RJ}u_{RJ}^* + u_{Rn}u_{Rn}^* + u_{Rl}u_{Rl}^* &=& 1, \\ 
\displaystyle \frac{M_R}{M_J} u_{RJ}^* u_{RJ} + \frac{M_R}{M_n} u_{Rn}^* u_{Rn}
+ \frac{M_R}{M_l} u_{Rl}^* u_{Rl} &=& 1.
\end{eqnarray*}
If we substract the first relation from the second one and multiply it by
$u_{RJ}^*$, we get
\begin{eqnarray*}
 0
&=& u_{RJ}(u_{RJ}^*)^2 - \frac{M_R}{M_J} u_{RJ}^* u_{RJ}u_{RJ}^* +
u_{Rn}(\underbrace{u_{RJ}u_{Rn}}_{=0})^* 
- \frac{M_R}{M_n}u_{Rn}^* (\underbrace{u_{Rn}u_{RJ}^*}_{=0}) 
\\ & & + \ 
u_{Rl}(\underbrace{u_{RJ}u_{Rl}}_{=0})^* - \frac{M_R}{M_l}u_{Rl}^*
(\underbrace{u_{Rl}u_{RJ}^*}_{=0}),
\end{eqnarray*}
since, by \eqref{eq_R_star_adjoint}, one gets 
\begin{eqnarray*}
 u_{Rn}u_{RJ}^* &=& 
\frac{\bar{E}_{Rik}}{\bar{E}_{Jnl}} \underbrace{u_{Rn}u_{in}} u_{kl}+
\frac{\bar{E}_{Rki}}{\bar{E}_{Jnl}} \underbrace{u_{Rn}u_{kn}} u_{il}
\stackrel{\eqref{eq_part_isom_basic}}{=}0 \quad \mbox{whenever}\quad n\neq J. 
\end{eqnarray*}
Therefore (keeping in mind that $M_k>0$) one gets
\begin{equation} \label{eq_pi_norm_square}
(u_{RJ})^2u_{RJ}^* = \frac{M_R}{M_J} u_{RJ} u_{RJ}^* u_{RJ} =
\frac{M_R}{M_J} u_{RJ}.
\end{equation}
Similarly, multiplying the same expression by $u_{RJ}$ (from the right)
we get
$u_{RJ}^* (u_{RJ})^2=\frac{M_R}{M_J}u_{RJ}u_{RJ}^*u_{RJ}$ and thus
\begin{equation} \label{eq_pi_almost_normal}
 u_{RJ}^* (u_{RJ})^2 = (u_{RJ})^2u_{RJ}^* .
\end{equation}
Finally, combining these results we show that
$$ u_{RJ}u_{RJ}^* \stackrel{\eqref{eq_pi_norm_square}}{=}
\frac{M_J}{M_R} (u_{RJ}^2u_{RJ}^*)u_{RJ}^*
= \frac{M_J}{M_R} u_{RJ} \big(u_{RJ}(u_{RJ}^*)^2 \big)
\stackrel{\eqref{eq_pi_almost_normal}}{=}
\frac{M_J}{M_R} u_{RJ} (u_{RJ}^*)^2 u_{RJ}
\stackrel{\eqref{eq_pi_norm_square}}{=}
u_{RJ}^*u_{RJ},$$
i.e\ $u_{RJ}$ is normal.
\end{proof}

\subsection{Commutation with adjoint}
We add to the list of relations, the commutation relation between
elements standing in two different rows and two different columns. This relation
is true for elements standing at the position for which a condition on
characteristic constants holds. 

\label{ssec_cross_rel}
\begin{lemma} \label{lem_ij_rn_star}
Let $\G$ be a PW-quantum group and let us fix the triples $(r,i,k),
(n,j,l) \in S_3$ for which $\co{k}{i}=1$ and $\co{j}{l}=1$. Then
\begin{equation} \label{eq_ij_rn_star}
u_{ij}u_{rn}^* = A_{i,k}^{j,l} u_{rn}^*  u_{ij}, 
\end{equation}
where $A_{i,k}^{j,l}$ is the non-zero constant that appears in Equation
\eqref{eq_comm_rel_1}.
\end{lemma}

\begin{proof}
If $\co{k}{i}=1$ and $\co{j}{l}=1$, then according to \eqref{eq_comm_rel_1} and
Corollary \ref{cor_comm_rel}, the following commutation relation holds
 $$ u_{kl} u_{ij} =  (A_{i,k}^{j,l} )^{-1} u_{ij} u_{kl}  -(A_{i,k}^{j,l}
)^{-1}B_{i,k}^{j,l} u_{kj} u_{il}$$
with
\begin{equation*}
 A_{i,k}^{j,l} = \frac{E_{rik}}{E_{rki}}\frac{E_{nlj}}{E_{njl}}
\frac{ 1+ \big|\frac {E_{rki}}{E_{rik}}
\big|^2}{1+\big|\frac{E_{nlj}}{E_{njl}}\big|^2} \neq 0, 
  \quad 
 B_{i,k}^{j,l} = \frac{E_{rik}}{E_{rki}}
\frac{1 - \big|\frac{E_{nlj}}{E_{njl}}\big|^2\big|\frac{E_{rki}}{E_{rik}}\big|^2
}{1+\big|\frac{E_{nlj}}{E_{njl}} \big|^2}.
\end{equation*}
Moreover, the condition on the characteristic constants 
imply, cf. Equation \eqref{eq_comm_rel_row1}, that 
$$ u_{al} u_{aj} = -\frac{E_{njl}}{E_{nlj}} u_{aj}u_{al} , \quad 
u_{ka}u_{ia} = -\frac{\bar{E}_{rik}}{\bar{E}_{rki}} u_{ia}u_{ka}.$$

Thus, by \eqref{eq_R_star} we have
\begin{eqnarray*}
u_{rn}^*  u_{ij}&\stackrel{\eqref{eq_R_star}}{=}&  \frac{E_{njl}}{E_{rki}}\
u_{kj} \underbrace{u_{il} u_{ij}}
+ \frac{E_{nlj}}{E_{rki}}\ \underbrace{u_{kl}u_{ij}} u_{ij}
\\ &\stackrel{\eqref{eq_Q_row}}{=}& 
-\frac{E_{njl}}{E_{rki}}\frac{E_{njl}}{E_{nlj}}\ \underbrace{u_{kj}u_{ij}}
u_{il} 
+ \frac{E_{nlj}}{E_{rki}}\ \left( (A_{i,k}^{j,l} )^{-1} u_{ij} u_{kl} 
-(A_{i,k}^{j,l} )^{-1}B_{i,k}^{j,l} u_{kj} u_{il} \right) u_{ij}
\\ &\stackrel{\eqref{eq_Q_col}}{=}& 
\frac{E_{njl}}{E_{rki}}\frac{E_{njl}}{E_{nlj}}\frac{\bar{E}_{rik}}{\bar{E}_{rki}
}\
u_{ij}u_{kj} u_{il} 
+ \frac{E_{nlj}}{E_{rki}} (A_{i,k}^{j,l} )^{-1}  u_{ij} \left(
\frac{E_{rki}}{E_{nlj}}\ u_{rn}^* -
 \frac{E_{njl}}{E_{nlj}}\ u_{kj} u_{il}\right) \\
& & + \frac{E_{nlj}}{E_{rki}}\frac{E_{njl}}{E_{nlj}}(A_{i,k}^{j,l}
)^{-1}B_{i,k}^{j,l} u_{kj}  u_{ij} u_{il} 
\\ &\stackrel{}{=}& 
(A_{i,k}^{j,l} )^{-1} u_{ij} u_{rn}^* + 
\frac{E_{njl}}{E_{rki}} \underbrace{\left( \frac{E_{njl}}{E_{nlj}}
\frac{\bar{E}_{rik}}{\bar{E}_{rki}} - (A_{i,k}^{j,l} )^{-1} 
- \frac{\bar{E}_{rik}}{\bar{E}_{rki}} (A_{i,k}^{j,l} )^{-1}B_{i,k}^{j,l} \right)
}_{M}
u_{ij} u_{kj} u_{il}. 
\end{eqnarray*}
However, 
\begin{eqnarray*}
M 
&=& \frac{E_{njl}}{E_{nlj}}\frac{\bar{E}_{rik}}{\bar{E}_{rki}} -
(A_{i,k}^{j,l} )^{-1} 
- \frac{\bar{E}_{rik}}{\bar{E}_{rki}} (A_{i,k}^{j,l} )^{-1}B_{i,k}^{j,l} 
\\ &=& 
\frac{( A_{i,k}^{j,l} )^{-1}}{1+\big|\frac{E_{nlj}}{E_{njl}} \big|^2} 
\left|\frac{E_{rik}}{E_{rki}}\right|^2
\left( 
\big|\frac {E_{rki}}{E_{rik}} \big|^2
- \left| \frac{E_{rki}}{E_{rik}} \right|^2 
(1+\big|\frac{E_{nlj}}{E_{njl}}\big|^2)
+ \big|\frac{E_{nlj}}{E_{njl}}\big|^2\big|\frac{E_{rki}}{E_{rik}}\big|^2
 \right) =0,
\end{eqnarray*} 
so the relation \eqref{eq_ij_rn_star} follows.
\end{proof}

\section{Definitions of quantum groups $U_q(2)$, $A_{p,k,m}(3)$ and
$SU_{p,m}$}
\label{sec_definitions}
In this section we give definitions and descriptions of three PW-quantum groups
for the construction in dimension $N=3$. To the best of our knowledge, apart
from $U_q(2)$, these quantum groups are not yet known. 
\subsection{Quantum group $U_q(2)$}
The quantum group $U_q(2)$ has already appeared in the literature. The
Hopf-algebraic version of the definition was given for example in
\cite{muller_hoissen93} (for $q\in \C^*$) as well as in the
unpublished preprint \cite{bichon99} (for $q\in \R^*$). The C$^*$-algebraic
version of the quantum group $U_q(N)$ for $q \in (0,1)$ and $N\geq 2$ was
provided by Koelink in \cite{koelink91}.

Moreover, for $q \in (0,1)$, Wysocza\'{n}ski \cite{wysoczanski04} constructed
$U_q(2)$ using the Woronowicz construction and described its irreducible
representations and corepresentations, and the related multiplicative unitary.
Independently, $U_q(2)$ for $q\in \C^*$ was studied by Zhang and Zhao in 
\cite{zhang_zhao_2005} and in \cite{zhang_2006}, who provided classification of
the irreducible representations in this wider context and gave an explicit
formula for the Haar state. 

\begin{defi} \label{def_uq2}
Let $q=|q|e^{\i t}\in \C^*$.  The quantum group $U_q(2)$ is the pair
$(\mathsf{A}_q,u)$, where $\mathsf{A}_q$ is the C${}^*$-algebra generated by
$a,c,v$ satisfying
the relations:
\begin{equation*}
ac=\bar{q} ca, \quad ac^*=q c^*a, \quad a^*a+c^*c=1, \quad aa^*
+|q|^2 c^*c=1, \quad cc^*=c^*c
\end{equation*}
and 
\begin{equation*}
vv^*=v^*v=1, \quad av=va, \quad cv=e^{-2\I t}vc,
\end{equation*}
and $u$ is the matrix
$$
u = \
\left ( \begin{array}{ccc}
v & 0 & 0 \\
0 & a & q c^*v^* \\
0 & c & a^*v^*
\end{array} \right ).$$
\end{defi}

The quantum group structure is imposed by the fact that $u$ is the fundamental
corepresentation. Hence 
\begin{eqnarray*}
 & & \Delta(a) = a \otimes a + q c^*v^*\otimes c, \quad \Delta(v)=v\otimes v,
\quad \Delta(c) = c\otimes a + a^*v^*\otimes c, \\
& & \e(a)=\e(v)=1, \quad \e(c)=0, \\
& & S(a) = a^*, \ S(c) = \bar{q}vc, \ S(v)=v^*, \ S(a^*)=a,\ 
S(v^*)=v, \ S(c^*)=\frac1{q} v^*c^*.
\end{eqnarray*}

It follows from \cite{wysoczanski04} that, when $q \in (0,1)$ or more general
when $q\in \R^*$, the quantum group
$U_{q}(2)$ admits a decomposition into the twisted product of $SU_{q}(2)$ and
$U(1)$, due to the fact that the element $v$ is central in the algebra
$\mathsf{A}_q$. When $q\not\in \R$, $v$ is no longer central and thus no evident
decomposition into subgroups holds.

The representation theory of the algebras $\A_q=C(U_q(2))$ differs according to
whether $|q|=1$, $|q|< 1$ or $|q|>1$. However, $U_q (2)$ is isomorphic to
$U_{q'}(2)$, where $q' = \frac1{\bar{q}}$, so the last two cases are dual. When
$q \in (0,1)$ there are two families of irreducible representations: 
the 1-dimensional family sending $c$ to 0, $a$ and $v$ to numbers of modulus 1,
and the $\infty$-dimensional family under which $a$ becomes a weighted
shift, while $c$ and $v$ are diagonal, cf. \cite[Theorem 3.1]{wysoczanski04}.
For $|q|=1$, the irreducible representations (apart from the 1-dimensional
ones) are related to irreducible representations of non-commutative tori
(rotation algebra), see \cite[Theorem 3.4]{zhang_zhao_2005}. 

We shall see in Theorem \ref{thm_uq2_or_t2} that the quantum group $U_q(2)$ is a
PW-quantum group related, for example, to the array 
\begin{eqnarray*} 
&E_{123}=1, \; E_{132}=-q, \; E_{213}=\alpha, & \\ 
&E_{231}=\beta, \; E_{312}=-\alpha\bar{q}, \; E_{321}=-\beta q &
\end{eqnarray*}
for any parameters $\alpha, \beta \in \C^*$. 

%
%

\subsection{Quantum groups $SU_{p,m}(3)$}
\label{ssec_ex_s_k}

In this section we introduce generalized $SU_q(3)$ groups and give their
basic properties. 

\begin{defi}[Quantum groups $SU_{p,m}(3)$] \label{def_su_k}
 Let $p\in \C^*$, $m\in \{0,1,2\}$, and set $\zeta=e^{\frac23 \i
\pi}$. We define $SU_{p,m}(3)$ to be the PW-quantum group related to the
array $E$ given by 
\begin{eqnarray*}
 E_{123}=1, &E_{132}=p, & E_{213}=p\zeta^{-m},\\
E_{231}=|p|^2 \zeta^m, & E_{312}=|p|^2 \zeta^{-m}, & E_{321}=|p|^2p\zeta^m. 
\end{eqnarray*}
\end{defi}

It immediately follows from the definition that when $p\in \R^*$ and $m=0$, 
the group $SU_{p,0}(3)$ is the twisted $SU_{-p}(3)$ group defined by Woronowicz
\cite{woronowicz88}. In particular, the commutation relations
\eqref{eq_comm_rel_row1} and \eqref{eq_comm_rel_1} agree with the relations
$(3)$ in \cite{bragiel89}. This is no longer the case for complex $p$ or $m\neq
0$. 

Computing the modular constants, we observe that the quantum group $SU_{p,m}(3)$
is of Kac type if and only if $|p|=1$. 

The most fundamental question concerning the quantum group $SU_{p,m}(3)$ is
whether it is non-trivial (i.e.\ if the related C$^*$-algebras are
noncommutative). For $|p|\neq 1$ the non-triviality is immediate, because
quantum versions of classical groups are always of Kac type. But the
argument presented below proves the non-triviality in the general case
$p\not\in \{-1,0,1\}$; namely, we show that, for any $m$, $SU_{p,m}(3)$
contains the (non-trivial) quantum group $U_{-p}(2)$ as a subgroup. This
justifies the interest in studying $SU_{p,m}(3)$. 

\medskip
Let us recall (cf. \cite{podles95}) that a compact quantum group $\mathbb{H}$
is a quantum subgroup of $\G$ if there exists $\Gamma:
\A(\G) \to \A(\mathbb{H})$, a surjective $*$-homomorphism, such that
$(\Gamma \otimes \Gamma) \circ \Delta_G = \Delta_\mathbb{H} \circ \Gamma$. 

\begin{prop} \label{prop_uq_in_suq}
 For any $p\in \C^*$ and $m\in \{0,1,2\}$, the quantum group
$U_{-p}(2)$ is a quantum subgroup of $SU_{p,m}(3)$.
\end{prop}

\begin{proof}
 Let us define the mapping from $\A\big(SU_{p,m}(3)\big)$ to $\A(U_{-p}(2))$ by
sending the coefficients of the fundamental corepresentation $u$ of
$SU_{p,m}(3)$ into the respective terms in the fundamental corepresentation of
$U_{-p}(2)$: 
 $$ \Gamma : \left ( \begin{array}{ccc}
u_{11} & u_{12} & u_{13}\\
u_{21} & u_{22} & u_{23} \\
u_{31} & u_{32} & u_{33} 
\end{array} \right ) \mapsto 
\left ( \begin{array}{ccc}
v & 0 & 0\\
0 & a & pc^*v^* \\
0 & c & a^*v^* 
\end{array} \right ).$$
It is a direct calculation to check that $\Gamma$ preserves the relations in
$SU_{p,m}(3)$ and thus extends to a $*$-homomorphism. It is obviously
surjective and, due to preservation of fundamental corepresentation, preserves
the coproduct (also the counit and the antipode). Thus $U_{-p}(2)$ is a quantum
subgroup of
$SU_{p,m}(3)$.
\end{proof}

\subsection{Quantum groups $A_{p,k,m}(3)$}
\label{ssec_ex_a}
In this section we define one more family of PW-quantum groups. 
\begin{defi}[Quantum groups $A_{p,k,m}(3)$] \label{def_akm}
Let $p\in \C^*$, $\zeta=e^{\frac{2}{3}\i\pi}$, $k,m\in \{0,1,2\}$.
We define $A_{p,k,m}(3)$ to be the PW-quantum group related to
the array $E$ given by 
\begin{equation*} 
E_{123}=1, \ E_{132}=p, \ E_{213}= p \zeta^k, \ 
E_{231}=\zeta^m,\
E_{312}=\zeta^{-m}, \ E_{321}=p\zeta^{-k} .
\end{equation*}
\end{defi}

Let us first observe that the quantum group $A_{p,k,m}(3)$ is always of Kac
type, because of the modular constants $M_j=1+|p|^2$ for $j=1,2,3$, and that
$u$ cannot be immediately decomposed, since the diagonal constants are all
equal 
$p_1=p_2=p_3=\zeta^m+|p|^2\zeta^{-k}$. Moreover, 
\begin{equation} \label{eq_akm_char_const}
 \co{1}{2}=\co{2}{3}=\co{3}{1}=\zeta^{-(k+m)}.
\end{equation}
Therefore, when $k+m\in 3N$, all characteristic constants are 1, whereas when
$k+m\not\in 3\N$, all characteristic constants differ from 1. The nature of
the quantum group $A_{p,k,m}(3)$ will differ accordingly.
We will show that in both cases ($k+m$ equals 0 modulo 3, or not), one can
choose parameters to ensure that $A_{p,k,m}(3)$ is non-trivial. 

\medskip
Let us focus on the case when $k+m\not\in 3\N$. The fact that $\co{n}{l}\neq 1$
for any $n\neq l$ strongly influences the commutation relations between
generators. In fact the following general result (to which we shall refer
later on as well) holds true for PW-quantum groups.  

\begin{lemma} \label{lem_akm_relations}
 Let $\G=(\A, u)$ be a PW-quantum group for which the characteristic constants
are all different from 1. Then for any $r\neq k,\, n\neq l$ we have
\begin{eqnarray}
\label{eq_s3_col} & & u_{rn} u_{kn} =0, \quad  u_{rn}u_{kn}^*=0, \quad
u_{rn}^*u_{kn}=0, \\ 
\label{eq_s3_row} & & u_{rn} u_{rl}=0, \quad u_{rn}^*u_{rl}=0, \quad
u_{rn}u_{rl}^*=0, 
\end{eqnarray}
and there exist complex
constants $A_{r,k}^{n,l}$, $B_{r,k}^{n,l}$ and $C_{r,k}^{n,l}\neq 0$ such that 
\begin{eqnarray}
\label{eq_s3_mix} & &u_{rn} u_{kl} = A_{r,k}^{n,l} u_{kl} u_{rn} + B_{r,k}^{n,l}
u_{kn} u_{rl}, \quad u_{rn}^*u_{kl} = C_{r,k}^{n,l} u_{kl}u_{rn}^* .
\end{eqnarray}
\end{lemma}

\begin{proof}
According to Theorem \ref{thm_comm_rel}, for $r\neq k,\, n\neq l$ the following
relations hold
\begin{equation*} 
 u_{rn} u_{kn} =0, \quad u_{rl} u_{rn}=0 \quad \mbox{and} \quad 
u_{rn} u_{kl} = A_{r,k}^{n,l} u_{kl} u_{rn} + B_{r,k}^{n,l} u_{kn} u_{rl},
\end{equation*}
where 
$$A_{r,k}^{n,l} = \frac{E_{irk}}{E_{ikr}} \frac{E_{jln}}{E_{jnl}} \frac{1 -
 \co{k}{r}}{1-\co{l}{n}} \quad \mbox{and} \quad 
 B_{r,k}^{n,l} = \frac{E_{irk}}{E_{ikr}} \frac{1 -
\co{k}{r}\co{l}{n}}{1-\co{l}{n}} .$$
Due to Theorem \ref{thm_pi}, we deduce that each $u_{rn}$ is a normal
partial isometry. Moreover, using \eqref{eq_R_star} and
\eqref{eq_comm_rel_row0}, we easily show (as in the proof of Theorem
\ref{thm_pi})
that 
\begin{equation*} 
 u_{kn}u_{rn}^*=0, \quad u_{rn}^*u_{kn}=0, \quad \mbox{for} \quad r\neq k.
\end{equation*}
Applying the antipode, cf.\ \eqref{eq_antipode}, we find out that
\begin{equation*} 
 u_{rn}u_{rl}^*=0, \quad u_{rn}^*u_{rl}=0, \quad \mbox{for} \quad n\neq l.
\end{equation*}
(Let us note that an alternative proof of this fact can be found in
\cite[Proposition 3.2]{banica_skalski}.)
Finally, we observe that for $r\neq k$ and $n\neq l$ we have
\begin{equation*} 
  u_{rn}^* u_{kl} = C_{r,k}^{n,l} \ u_{kl} u_{rn}^*, \quad \mbox{where} \quad
C_{r,k}^{n,l} = \frac{E_{rik}}{E_{rki}} \frac{E_{nlj}}{E_{njl}} \neq 0,
\end{equation*}
because of
\begin{eqnarray*}
 u_{rn}^* u_{kl} &=& \frac{E_{nlj}}{E_{rki}}\
 u_{kl}u_{ij}u_{kl} + \frac{E_{njl}}{E_{rki}}\ u_{kj}
\underbrace{u_{il}u_{kl}}_{=0} 
=\ \frac{E_{nlj}}{E_{rki}}u_{kl}\left(
\frac{E_{rik}}{E_{njl}}\ u_{rn}^*-
\frac{E_{nlj}}{E_{njl}}\ u_{il}u_{kj}  \right) .
\end{eqnarray*}
\end{proof}

Let us now consider a representation $\pi$ of $\A=C(A_{p,k,m}(3))$
on some Hilbert space $H$ and denote $U_{ij}=\pi(u_{ij})$ for $i,j=1,2,3$. The
next lemma shows that the range of each $U_{ij}$ is a $\pi$-invariant subspace.
Again, it will be convenient to formulate the result in a more general way. 

\begin{lemma} \label{lem_akm_invariance}
Let $\G=(\A, u)$ be a compact quantum group of Kac type.
Assume that there exist complex constants $A_{r,k}^{n,l}$, $B_{r,k}^{n,l}$ and
$C_{r,k}^{n,l}\neq 0$ such that for $r\neq k,\, n\neq l$ the relations
\eqref{eq_s3_col} - \eqref{eq_s3_mix} hold true. 
Then for any representation $\pi$ of $\A=C(\G)$ on some Hilbert space $H$ the
subspace $\pi(u_{ij})(H)$ is invariant.
\end{lemma}
\begin{proof}
Let us consider a representation $\pi$ of $\A=C(\G)$
on some Hilbert space $H$ and denote $U_{ij}=\pi(u_{ij})$ and
$K_{ij}=U_{ij}U_{ij}^*(H)$ for $i,j=1,2,3$. 

The relations in $\G$ together with Theorem \ref{thm_pi} implies that $u_ij$ is
a normal partial isometry, hence $U_{ij}U_{ij}^*$ is the projection onto the
range of $U_{ij}$, thus
$K_{ij}=U_{ij}(H)$. It remains to show that $K_{ij}$ is an invariant subspace.

It follows from \eqref{eq_s3_row} and the normality of $u_{ij}$ that
\begin{eqnarray*}
 u_{in}(u_{ij}u_{ij}^*) = (u_{ij}u_{ij}^*)u_{in}, &
u_{in}^*(u_{ij}u_{ij}^*) = (u_{ij}u_{ij}^*)u_{in}^* & \mbox{for}\;
n=1,2,3,
\end{eqnarray*}
and similarly, from \eqref{eq_s3_col} and the normality of $u_{ij}$, we deduce
that $u_{rj}^\epsilon(u_{ij}u_{ij}^*) =
(u_{ij}u_{ij}^*)u_{rj}^\epsilon$ for $\epsilon\in \{1,*\}$. On the other hand,
using \eqref{eq_s3_col}-\eqref{eq_s3_mix}, for $r\neq i$
and $n\neq j$ we have
\begin{eqnarray*}
 u_{rn}(u_{ij}u_{ij}^*) &=& 
 A_{i,r}^{j,n} u_{ij}u_{rn}u_{ij}^* + B_{i,r}^{j,n} u_{in}
\underbrace{u_{rj} u_{ij}^*}_{=0}
=A_{i,r}^{j,n}(C_{i,r}^{j,n})^{-1} (u_{ij}u_{ij}^*) u_{rn}, 
\\ 
u_{rn}^* (u_{ij}u_{ij}^*) & \stackrel{\eqref{eq_s3_mix}}{=} & 
C_{r,i}^{n,j} u_{ij}\left( A_{i,r}^{j,n} u_{rn} u_{ij} + B_{i,r}^{j,n} u_{rj}
u_{in} \right)^* 
\\ &=& C_{r,i}^{n,j} \bar{A}_{i,r}^{j,n} u_{ij} u_{ij}^* u_{rn}^* +
C_{r,i}^{n,j} \bar{B}_{i,r}^{j,n} u_{ij} u_{in}^* u_{rj}^* 
\stackrel{\eqref{eq_s3_row}}{=} C_{r,i}^{n,j} \bar{A}_{i,r}^{j,n} 
(u_{ij}u_{ij}^*) u_{rn}^*.
\end{eqnarray*}
So whenever $x\in K_{ij}$, we have $u_{rn}^\epsilon x\in K_{ij}$. Since the
coefficients of the fundamental corepresentation are algebraic generators, this
shows that $K_{ij}$ is $\pi$-invariant. 
\end{proof}

Let us recall that the rotation algebra $A_\theta$ (called also
noncommutative torus) is the universal unital C$^*$-algebra generated by two
unitaries $v_1$ and $v_2$ subject to relation $v_1v_2=e^{2\i \pi \theta}v_2v_1$
(cf. \cite{anderson_paschke89} or \cite{williams07}).

\begin{thm} \label{thm_akm_reps}
 There are three families of irreducible representations of
$\A=C(A_{p,k,m}(3))$ with $k+m\not\in 3\mathbb{N}$:
 \begin{enumerate}
  \item one-dimensional representations: $ \pi(u_{11})=z_1$, $
\pi(u_{22})=z_2$, $\pi(u_{33})=\bar{z_1}\bar{z_2}$ and $\pi(u_{ij})=0$ 
for $i\neq j,$
where $|z_1|=|z_2|=1$.
  \item representation defined by
  $$ \pi(u_{12})=\pi_\theta (v_1), \quad \pi(u_{23})=\pi_\theta (v_2),
\quad \pi(u_{31})=\zeta^{-m}\pi_\theta (v_2^*v_1^*)$$
and $\pi(u_{ij})=0$ otherwise,
where $\pi_\theta$ is an irreducible representation of the (rational) rotation 
algebra $A_\theta$ with the generators $v_1$ and $v_2$, and $\theta =
\frac{k-m}{3}$.
  \item representation defined by
  $$ \pi(u_{13})=\pi_{\theta'} (w_1), \quad\pi(u_{21})=\pi_{\theta'} (w_2),
\quad \pi(u_{32})=\zeta^{m}\pi_{\theta'} (w_2^*w_1^*)$$
and $\pi(u_{ij})=0$ otherwise, where $\pi_{\theta'}$ is an irreducible
representation of the (rational) 
rotation algebra $A_{\theta'}$ with the generators $w_1$ and $w_2$, and
${\theta'} =\frac{m-k}{3}$.
 \end{enumerate}
\end{thm}

\begin{proof}
Let $\pi$ be an irreducible representation of $\A=C(A_{p,k,m}(3))$ on a
Hilbert space $H$ and denote $U_{ij}:=\pi(u_{ij})$. Since in $A_{p,k,m}(3)$ the
relations \eqref{eq_s3_col} - \eqref{eq_s3_mix} from Lemma
\ref{lem_akm_invariance} are satisfied, hence each subspace $K_{ij}=U_{ij}(H)$
is $\pi$-invariant, thus trivial (i.e.\ $\{0\}$ or $H$). 

If $K_{ij}=\{0\}$ for some $i$ and $j$, then $U_{ij}=0$. If $K_{ij}=H$, then
$U_{ij}$ is unitary. Moreover, the relations (U1) and (U2) imply that there is
exactly one non-zero unitary element in each row and each column of $u$. This
means
that for the representation $\pi$ there exists a permutation $\sigma$ such that
$\pi(u_{k,\sigma(k)})\neq 0$ and $\pi(u_{k,j})= 0$ if $j\neq \sigma(k)$. Let us
denote these non-zero elements by $V_k=\pi(u_{k,\sigma(k)})$ and let for
simplicity denote $k':=\sigma(k)$.

We shall show that only specific permutations can appear in the description
of $\pi$. For that, we use the condition (TD) to express the commutation
relations between $V_k$'s. Namely, we have
\begin{eqnarray*}
 E_{k',r',i'} V_k V_r V_i = E_{kri} 1
 &\quad\Rightarrow \quad &
 V_i^* = \frac{E_{k',r',i'}}{E_{kri}} V_k V_r
 \\
 E_{i',k',r'} V_i V_k V_r = E_{ikr}1
 &\quad\Rightarrow \quad & 
 V_i^* = \frac{E_{i',k',r'}}{E_{ikr}} V_k V_r .
\end{eqnarray*}
Since $V_k$ and $V_r$ are invertible, the equations above lead to 
\begin{equation} \label{eq_akm_perm_cond}
\frac{E_{k',r',i'}}{E_{kri}} = \frac{E_{i',k',r'}}{E_{ikr}}, 
\end{equation}
which must be satisfied for any triple $(k,r,i)$. Moreover, if
\eqref{eq_akm_perm_cond} is satisfied, then the relation $E_{r',k',i'} V_r V_k
V_i = E_{rki} 1$ implies that any two of $V_k$'s commute up to a non-zero
constant 
\begin{equation} \label{eq_akm_comm_v1v2}
 V_k V_r =
\frac{E_{r',k',i'}}{E_{k',r',i'}} \frac{E_{kri}}{E_{rki}} V_r V_k.
\end{equation}

Let us check when the condition \eqref{eq_akm_perm_cond} holds. Obviously, it
remains true for $\sigma =\id $ (an identity permutation). In this case, the
non-zero elements appear on the diagonal of $\pi(u)$ and $V_k$'s pairwise
commute (the constant equals 1). Therefore for any
$k=1,2,3$ there exists $z_k$ of modulus 1 such that $V_k=z_k I_H$. Moreover,
$V_3= V_1^*V_2^*$, and the Hilbert space must be one-dimensional (otherwise, two
orthogonal vectors span two invariant subspaces). We conclude that in this
case the representation $\pi$ is one-dimensional and given by
$$ \pi(u_{11})=z_1, \quad \pi(u_{22})=z_2, \quad
\pi(u_{33})=\bar{z_1}\bar{z_2}, \quad \pi(u_{ij})=0 \;\; \mbox{for} \; i\neq
j,$$
where $|z_1|=|z_2|=1$.

On the other hand, it is immediate to see that $\sigma$ cannot be a
transposition. Indeed, if $\sigma$ sent the indices as follows: $r\to k$, $k\to
r$, $i\to i$, then we would have
$$\co{r}{k}=\frac{E_{ikr}}{E_{irk}}\frac{E_{rki}}{E_{kri}} =
\frac{E_{ikr}}{E_{i',k',r'}} \frac{E_{k',r',i'}}{E_{kri}}
\stackrel{\eqref{eq_akm_perm_cond}}{=} 1,$$ 
which would contradict Equation \eqref{eq_akm_char_const}, since
$\zeta^{\pm(k+m)}\neq 1$.

Finally, we can check by direct calculation that the condition
\eqref{eq_akm_perm_cond} is satisfied for the permutations $\sigma=(231)$ and
$\sigma =(312)$. In the first case, the related unitaries satisfy  
$$V_3= \zeta^{-m} V_2^*V_1^*, \quad V_1V_2= 
e^{\frac{2}{3}
(k-m) \i\pi} V_2V_1 ,$$
hence the representation $\pi$ must be of the form 
$$ \pi(u_{12})=\pi_\theta (v_1), \quad \pi(u_{23})=\pi_\theta (v_2), \quad
\pi(u_{31})=e^{-\frac{2\pi}{3} m \i}\pi_\theta (v_2^*v_1^*) , \quad
\pi(u_{ij})=0 \; \mbox{otherwise},$$
where $\pi_\theta$ is an irreducible representation of the (rational) rotation 
algebra $A_\theta$ with the generators $v_1$ and $v_2$, and $\theta =
\frac{k-m}{3}$. 
In the other case, when $\sigma =(312)$, the non-zero elements of $\pi(u)$
satisfy
$$ V_3= \zeta^{m} V_2^*V_1^*, \quad V_1V_2= 
e^{\frac{2}{3} (m-k) \i\pi} V_2V_1 .$$
Thus the representation $\pi$ is of the form 
$$ \pi(u_{13})=\pi_{\theta'} (w_1), \quad \pi(u_{21})=\pi_{\theta'} (w_2), \quad
\pi(u_{32})=e^{\frac{2}{3} m \i\pi}\pi_{\theta'} (w_2^*w_1^*) , \quad
\pi(u_{ij})=0 \; \mbox{otherwise},$$
where $\pi_{\theta'}$ is an irreducible representation of the (rational)
rotation algebra $A_{\theta'}$ with the generators $w_1$ and $w_2$, and
${\theta'} =\frac{m-k}{3}=-\theta$.
This finishes the proof. 
\end{proof}

We conclude that when $k\neq m$, then $C(A_{p,k,m}(3))$ contains
noncommutative elements and thus as a quantum group is non-trivial. 

\begin{rem} \label{rem_akm}
For $k=m$, the quantum $A_{p,m,m}(3)$ is trivial.
Indeed, the non-zero elements of the representations related to permutations
$\sigma=(231)$ and $\sigma =(312)$ must commute and thus are given by $V_j=z_j
I_H$ for some complex $z_j$ of modulus 1. Such representations must be
one-dimensional by the same argument as described for the identity permutation.
Therefore, the quantum groups algebra is commutative and $A_{p,m,m}(3) = C(G)$
for some (classical) compact group $G$. 

The group $G$ can be recovered once we realize that its points are
in one-to-one correspondence with the irreducible representations of
$\mathsf{A}=C(G)$ and that the group multiplication reflects the
comultiplication. More explicitly, set $\theta=\zeta^{-m}$ and let 
$e=(123)$, $x=(231)$ and $y=(312)$ denote the elements of the  alternating
group $A_3$. Define 
\begin{equation}
(z,w;e)=\left( \begin{array}{ccc} z & 0 & 0 \\ 0 & w & 0 \\ 0 & 0
&
\overline{zw} \end{array} \right), \;
(z,w;x)=\left( \begin{array}{ccc} 0 & z & 0 \\ 0 & 0 & w \\ 
\theta \overline{zw} & 0 & 0 \end{array} \right), \; 
(z,w;y)=\left( \begin{array}{ccc} 0 & 0 & z \\ w & 0 & 0 \\ 0 &
\bar{\theta }\overline{zw} & 0 \end{array} \right).
\end{equation}
Then $G=\{ (z,w;a); a\in A_3, z,w\in \T\},$ and the multiplication in $G$ is
defined by $ (z_1,w_1;e) \cdot (z_2,w_2;a)= (z_1z_2,w_1w_2;a)$ for any $a\in
A_3$ and,  by 
$$\begin{array}{ll}
 (z_1,w_1;x) \cdot (z_2,w_2;e)=  (z_1w_2,w_1\overline{z_2w_2};x), 
 &   (z_1,w_1;y) \cdot (z_2,w_2;e)=  (z_1\overline{z_2w_2}, w_1z_2;y),
  \\ 
 (z_1,w_1;x) \cdot (z_2,w_2;x)=  (z_1w_2,\theta w_1\overline{z_2w_2};y),
 &  (z_1,w_1;y) \cdot (z_2,w_2;x)=  (\theta z_1\overline{z_2w_2}, w_1z_2;e),
 \\
 (z_1,w_1;x) \cdot (z_2,w_2;y)=  (z_1w_2,\bar{\theta} w_1\overline{z_2w_2};e),
 &  (z_1,w_1;y) \cdot (z_2,w_2;y)=  (\bar{\theta} z_1\overline{z_2w_2},
w_1z_2;x).
 \end{array}$$
\end{rem}

\medskip
Let us now turn our attention to the case when $k+m\in 3\mathbb{N}$ or
$m\equiv -k ({\rm mod } 3)$. Then all $\co{n}{l}$'s equal 1 and the
quantum group $A_{p,-m,m}(3)$ is more complicated than in the previous case; we
shall not study it here in details. However, the same method as in the proof of
Theorem \ref{thm_akm_reps} allows to establish the existence of torus-like
representations of $A_{p,-m,m}(3)$.

\begin{prop}
Let $\sigma$ be one of the permutations $\sigma_a=(231)$ or $\sigma_b=(312)$ and
let $V_1$ and $V_2$ be two unitary operators on a Hilbert space $H$ satisfying 
$$V_1V_2 = c(\sigma) V_2V_1,$$
where $c(\sigma_a)=\zeta^{-m}$, $c(\sigma_b)=\zeta^{m}$. 
Then $\pi$ defined as 
 $$ \pi(u_{1\sigma(1)})=V_1, \quad \pi(u_{2\sigma(2)})=V_2, \quad
\pi(u_{3\sigma(3)})= \overline{c(\sigma)} V_2^*V_1^* , \quad
\pi(u_{ij})=0 \; \mbox{otherwise},$$
is a $*$-representation of $A_{p,-m,m}(3)$. 
\end{prop}

\begin{proof}
 Let $V_3=  \overline{c(\sigma)} V_2^*V_1^*$. The condition (U) is obviously
satisfied, since only one term in each row and each column is non-zero. Next we
verify that the compatibility condition \eqref{eq_akm_perm_cond} is satisfied in
our case (see the proof of Theorem \ref{thm_akm_reps}).
However, $V_1$ and $V_2$ are chosen in such a way that (TD),
i.e.\ $E_{\sigma(r),\sigma(k),\sigma(i)} V_r V_k V_i = E_{rki} 1$, holds for
any triple $(r,k,i)\in S_3$. 
\end{proof}

From the Proposition above we can conclude that, for $m\neq 0$, the quantum
group $A_{p,-m,m}(3)$ is non-trivial (as it contains a 'noncommutative'
representation). Moreover, it is immediate to note that for $|p|=1$ 
$A_{p,-m,m}(3)$ coincides with $SU_{p,m}(3)$. Then, by Proposition
\ref{prop_uq_in_suq}, $A_{p,0,0}(3)$ contains $U_{-p}(2)$ as a subgroup, and
thus is non-trivial when $p\neq \pm 1$. It remains to check
what happens when $m=0$ and $|p|\neq 1$. A partial answer to this question, with
additional assumption
that $p$ is real, is stated below. It turns out that the resulting quantum group
is again trivial. 

\begin{prop}
 If $p\in \R\setminus \{-1,0,1\}$ and $m=0$, then $A_{p,0,0}(3) \cong C(\T^2
\rtimes_\phi A_3)$, where $\T^2
\rtimes_\phi A_3$ denotes the semi-direct product of the
two-dimensional torus $\T^2$ and the alternating group $A_3$ (a subgroup of
$S_3$ composed of all even permutations). In particular, $A_{p,0,0}(3)$ is
trivial. 
\end{prop}

\begin{proof}
Recall that $A_{p,0,0}(3)$ is of Kac type and that we have
$\co{n}{l}=1$ for any $n\neq l$. Then
$\frac{E_{njl}}{E_{jln}}=\frac{E_{nlj}}{E_{ljn}}$ for any $(n,j,l)\in S_3$
and it makes sense to define $\omega_n:=\displaystyle \frac{E_{.{}_*
n}}{E_{n.{}_* }}$. One checks that, since $m=0$, $\omega_n=1$ for $n=1,2,3$. 

Moreover, it is easy to note that $\frac{E_{jln}}{E_{jln}}$
can take only two (real) values: $p$ or $\frac1{p}$. This implies that, in
$u_{rn} u_{kl} = A_{r,k}^{n,l} u_{kl} u_{rn} +B_{r,k}^{n,l} u_{kn} u_{rl}$, the
constant $A_{r,k}^{n,l}$  must equal 1 (check all four possible cases
$\frac{E_{irk}}{E_{ikr}}, \frac{E_{jln}}{E_{jnl}}\in \{p,\frac1{p}\}$).
Similarly, we verify that $B_{r,k}^{n,l}=\frac{E_{irk}}{E_{ikr}}
-\frac{E_{jln}}{E_{jnl}}.$
Consequently, the commutation relation \eqref{eq_comm_rel_1} takes the form 
\begin{equation} \label{eq_mix_special}
u_{rn} u_{kl} = \,u_{kl} u_{rn}  + \left(\frac{E_{irk}}{E_{ikr}}
-\frac{E_{jln}}{E_{jnl}} \right)\, u_{kn} u_{rl}.
\end{equation}
The main part of the proof is divided into three steps.

\medskip
\noindent {\bf Step 1.}
We aim to show that $u_{rn}u_{rj}^*=0$ and $u_{rn}^*u_{rj}=0$ for any $r$ and
$n\neq j$. For that, we write the orthogonality relation (O)
$$u_{rn}^* u_{rj} + u_{kn}^* u_{kj} + u_{in}^* u_{ij}=0 \qquad \mbox{for}
\quad n\neq j,$$
which holds because $A_{p,0,0}$ is of Kac type, and using \eqref{eq_R_star},
\eqref{eq_mix_special} and other commutation relations we rearrange the
terms in the same order $u_{k.}u_{i.}u_{r.}$. Therefore
\begin{eqnarray*}
 \lefteqn{ u_{rj}u_{rn}^* + u_{kj}u_{kn}^* + u_{ij}u_{in}^*} \\
&=& 
\frac{E_{njl}}{E_{rik}}\ u_{rj}u_{ij}u_{kl} +
\frac{E_{nlj}}{E_{rik}}\ u_{rj}u_{il}u_{kj}
 + \frac{E_{njl}}{E_{kri}}\ u_{kj}u_{rj}u_{il} \\ & &
 + \frac{E_{nlj}}{E_{kri}}\ u_{kj}u_{rl}u_{ij}
 +\frac{E_{njl}}{E_{irk}}\ u_{ij}u_{rj}u_{kl} 
 +  \frac{E_{nlj}}{E_{irk}}\ u_{ij}u_{rl}u_{kj} \\
&=& 
\frac{E_{njl}}{E_{rik}}\ u_{rj}u_{ij}u_{kl} +
\frac{E_{nlj}}{E_{rik}}\ u_{rj}u_{il}u_{kj}
 - \frac{E_{njl}}{E_{kri}} \frac{E_{irk}}{E_{ikr}}\ u_{rj} u_{il}u_{kj} 
+ \frac{E_{njl}}{E_{kri}}\frac{E_{irk}}{E_{ikr}}
\left( \frac{E_{rik}}{E_{rki}}- \frac{E_{njl}}{E_{nlj}} \right) u_{rj}
u_{ij}u_{kl} \\ & & 
 - \frac{E_{nlj}}{E_{kri}} \frac{E_{rik}}{E_{rki}} u_{rl}u_{ij}u_{kj} -
\frac{E_{nlj}}{E_{kri}}\left(
\frac{E_{irk}}{E_{ikr}}-
\frac{E_{njl}}{E_{nlj}} \right) 
\left( u_{rj}u_{ij}u_{kl} - \left( \frac{E_{rik}}{E_{rki}}-
\frac{E_{nlj}}{E_{njl}} \right) u_{rj}u_{il}u_{kj} \right)
\\ & & 
 -\frac{E_{njl}}{E_{irk}}\frac{E_{kri}}{E_{kir}}\ u_{rj}u_{ij}u_{kl}
 +  \frac{E_{nlj}}{E_{irk}}\ u_{rl}u_{ij}u_{kj}
- \frac{E_{nlj}}{E_{irk}}\left( \frac{E_{kri}}{E_{kir}} -
\frac{E_{njl}}{E_{nlj}} \right)
u_{rj}u_{il}u_{kj}.
\\&=& 
c_1 \, u_{rj}u_{ij}u_{kl} 
+ 
c_2\, u_{rj}u_{il}u_{kj} 
+ 
c_3 \, u_{rl}u_{ij}u_{kj}. 
\end{eqnarray*}
Due to the fact that $\omega_n=1$ for $n=1,2,3$, we immediately see that
$$c_3 := \frac{E_{nlj}}{E_{irk}}- \frac{E_{nlj}}{E_{kri}}
\frac{E_{rik}}{E_{rki}}
  = \frac{E_{nlj}}{E_{irk}} ( 1 - \frac{\omega_k}{\omega_i} )=0.
$$
On the other hand, with the notation 
\begin{equation*} \label{eq_assumption0}
x= \frac{E_{rik}}{E_{rki}} \quad \mbox{and} \quad
y=  \frac{E_{njl}}{E_{nlj}}, 
\end{equation*}
we have 
$$  \frac{E_{kri}}{E_{kir}}= \frac{E_{rik}}{E_{irk}} = \frac{E_{rik}}{E_{rki}}
\frac{E_{rki}}{E_{irk}} = x\omega_i = x \quad \mbox{and} \quad
 \frac{E_{irk}}{E_{ikr}}= \frac{E_{rki}}{E_{kri}} =
\frac{E_{rki}}{E_{rik}} \frac{E_{rik}}{E_{kri}} = \frac{\omega_k}{x}
= \frac{1}{x}.$$
This allows to compute the remaining two constants:
\begin{eqnarray*}
c_1 &:=&  
\frac{E_{njl}}{E_{rik}}+
\frac{E_{njl}}{E_{kri}}\frac{E_{irk}}{E_{ikr}}
\left( \frac{E_{rik}}{E_{rki}}- \frac{E_{njl}}{E_{nlj}} \right) 
- \frac{E_{nlj}}{E_{kri}}\left(
\frac{E_{irk}}{E_{ikr}}-\frac{E_{njl}}{E_{nlj}} \right) 
 -\frac{E_{njl}}{E_{irk}}\frac{E_{kri}}{E_{kir}}
\\ &=&  
\frac{E_{njl}}{E_{rik}}\big [ 1+ \frac{\omega_k}{x}
\left(x- y \right) - \frac{1}{y\omega_k }\left(
\frac{1}{x}-y \right)  -x^2 \big]
= - \frac{E_{njl}}{E_{rik}} \left(x-\frac{1}{x} \right)^2
\end{eqnarray*}
and
\begin{eqnarray*}
c_2 &:=&  \frac{E_{nlj}}{E_{rik}}
- \frac{E_{njl}}{E_{kri}} \frac{E_{irk}}{E_{ikr}}
 + \frac{E_{nlj}}{E_{kri}}\left( \frac{E_{irk}}{E_{ikr}}
 - \frac{E_{njl}}{E_{nlj}}\right) 
 \left( \frac{E_{rik}}{E_{rki}}-\frac{E_{nlj}}{E_{njl}} \right)
- \frac{E_{nlj}}{E_{irk}}\left( \frac{E_{kri}}{E_{kir}} -
\frac{E_{njl}}{E_{nlj}} \right)
\\ &=& 
 \frac{E_{nlj}}{E_{rik}} \big[ 1 - \frac{y \omega_k}{x}
 + \omega_k \left( \frac{1}{x} - y \right) \left( x -\frac{1}{y} \right)
- x \left( x  - y \right) \big] = 
-\frac{E_{nlj}}{E_{rik}} \left(x-\frac{1}{x} \right)^2.
\end{eqnarray*}
So finally, we get 
\begin{eqnarray*}
0 
&=&
-\left( x-\frac{1}{x} \right)^2 \left( \frac{E_{njl}}{E_{rik}}
u_{rj}u_{ij}u_{kl}  +
\frac{E_{nlj}}{E_{rik}} u_{rj}u_{il}u_{kj} \right) 
= -\left( p-\frac{1}{p} \right)^2 u_{rj} u_{rn}^*, 
\end{eqnarray*}
because $x=p$ or $x=\frac{1}{p}$. Since $p\neq \pm 1$ we conclude that
$u_{rj}u_{rn}^*=0$.

Starting from the unitarity relations:
$$u_{rn}^* u_{rj} + u_{kn}^* u_{kj} + u_{in}^* u_{ij}=0 \qquad \mbox{for}
\quad n\neq j,$$
we apply the same technique to prove that $u_{rn}^*u_{rj}=0$. 

\medskip
\noindent {\bf Step 2.} We will prove that the relations
\eqref{eq_s3_col}-\eqref{eq_s3_mix} hold. For let us observe that the
relations $u_{rn}u_{rj}^*=0$ and
$u_{rn}^*u_{rj}=0$ (for $n\neq j$), together with \eqref{eq_comm_rel_row1},
yield
$$ (u_{rn}u_{rj})(u_{rn}u_{rj})^* =
-\frac{E_{ljn}}{E_{lnj}} u_{rj}(u_{rn}u_{rj}^*)u_{rn}^*=0, $$
from which we deduce that $u_{rn}u_{rj}=0$.
By applying the antipode we get that 
\begin{equation*} \label{eq_rn_star_kn}
u_{rn}^*u_{kn} = u_{kn}u_{rn}^*=0 \quad \mbox{and} \quad u_{kn}u_{rn}=0  
\end{equation*}
for $r\neq k$. This, together with Theorem
\ref{thm_comm_rel}, Theorem \ref{thm_pi} and Lemma \ref{lem_ij_rn_star}, shows
that each $u_{rn}$ is a normal partial isometry and that 
\eqref{eq_s3_col}-\eqref{eq_s3_mix} are satisfied. 

\medskip
\noindent {\bf Step 3.} Let us now consider an irreducible $*$-representation
$\pi$ of $\A=C(A_{p,0,0}(3))$ on a Hilbert space $H$. Applying Lemma
\ref{lem_akm_invariance} we show that $K_{ij}=\pi(u_{ij})(H)=
\pi(u_{ij}u_{ij}^*)(H)$, $i,j=1,2,3$, is $\pi$-invariant, hence $\{0\}$ or
$H$. Repeating the reasoning from the proof of Theorem \ref{thm_akm_reps} we
realize that $\pi$ can be of the form as described therein, i.e.\ it is related
to an even permutation $\sigma\in S_3$ in the sense that the only non-zero
(unitary) elements are $\pi(u_{j,\sigma(j)})$. It remains to check if $\pi$ can
be related to a transposition. The compatibility condition
\eqref{eq_akm_perm_cond} is satisfied. Elementary calculations shows, however,
that if $V_1=\pi(u_{1,\sigma(1)})$ and $V_2=\pi(u_{2,\sigma(2)})$ was two
unitaries and $\sigma$ was a transposition, then \eqref{eq_akm_comm_v1v2} would
lead to $V_1V_2 =p^{-2} V_2 V_1$ which can never hold for unitaries due to the
assumption $|p|\neq 1$. 

We conclude that all irreducible representations of $A_{p,0,0}(3)$ are described
in Theorem \ref{thm_akm_reps}. However, since $m=k=0$, by the same reasoning
as in Remark \ref{rem_akm}, we see that the quantum groups algebra $\mathsf{A}$
is commutative and thus $A_{p,0,0}(3)$ is the quantum version of a classical
group $G$, described in Remark \ref{rem_akm} with $\theta=1$. Then, one can
observe that 
$G= \T^2 \rtimes_\phi A_3$, 
where $\phi: A_3 \to {\rm Aut}\, \T^2$ is defined by 
$$\phi_{e}(z,w)=(z,w), \quad \phi_{x}(z,w)=(w, \overline{zw}), \quad
\phi_{y}(z,w)=(\overline{zw},z).$$
Indeed, it is a direct calculation to check that $(z_1,w_1;a) \cdot (z_2,w_2;b)=
((z_1,w_1) \phi_a(z_2,w_2);ab)$ for any $a,b\in A_3$ and 
$(z_1,w_1),(z_2,w_2)\in \T^2$. 

This means that $\G$ is the quantum version of the semi-direct product of the
two-dimensional torus $\T^2$ and the
alternating group $A_3$.
\end{proof}

We collect the result on non-triviality of $A_{p,k,m}(3)$. 
\begin{cor}
Let $k,m\in \{0,1,2\}$. The quantum group $A_{p,k,m}(3)$ is non-trivial in the
following cases:
\begin{enumerate}
 \item for $k+m\not\in 3\N$ and $k\neq m$;
 \item for $m\neq 0$ and $k=-m$;
 \item for $k=m=0$ and $|p|=1$, $p\in \C\setminus \R$.
\end{enumerate}
\end{cor}

\section{Two-blocks decomposition}
\label{sec_decomposition}

In this Section we shall show that if the fundamental corepresentation $u$ of a
PW-quantum group admits a decomposition into two blocks, then the resulting
quantum group is either $U_q(2)$ or $C(\T^2)$. The proof makes no reference to
morphism
properties, since in this case it is possible to describe the
related quantum groups using explicit calculations. This method seems to be
more instructive and sheds more light on where the different relations (in the
general case) come from. In Remark \ref{rem_uq2_and_morphisms} we compare the
results obtained by direct calculations with these due to morphisms properties. 

\begin{thm} \label{thm_uq2_or_t2}
Let the fundamental corepresentation $u$ of a PW-quantum group $\G$ admits a
decomposition into two blocks. 
Then $\G \cong C(\T^2)$ or $\G\cong U_q(2)$. 

Moreover, $\G\cong U_q(2)$ if and only if
$\displaystyle \frac{E_{kri}}{E_{kir}} = \frac{E_{rik}}{E_{irk}}$ 
and $\displaystyle
\frac{E_{kri}}{E_{kir}}=\frac{\bar{E}_{ikr}}{\bar{E}_{rki}}$, where $k$
denotes the index of the 1-dimensional block and $r$ denotes
the first row of the $2\times 2$-block.
In this case, $\displaystyle q=-\frac{E_{kri}}{E_{kir}}\in \C^*$.
\end{thm}

\begin{proof}
Let us assume that the corepresentation matrix $u$  decomposes into a block of
dimension 1 and a block of dimension 2. Due to Lemma \ref{lem_iso}, we can
assume without loose of generality that $u$ is of the form
\begin{equation} \label{eq_u_decomp_form}
u = \left ( \begin{array}{ccc}
u_{11} & 0 & 0\\
0 & u_{22} & u_{23} \\
0 & u_{32} & u_{33} 
\end{array} \right )
:=
\left ( \begin{array}{ccc}
v & 0 & 0 \\
0 & a & b \\
0 & c & d
\end{array} \right ) = v\oplus \left ( \begin{array}{cc}
a & b \\
c & d
\end{array} \right ). 
\end{equation}

From the unitarity condition (U) of $u$ we have the following relations
$$\begin{array}{cccccccc}
{\rm (A)} & aa^*+bb^*=1 & \quad & {\rm (B)} & b^*b+d^*d=1 &
\quad & {\rm (C)} & cc^*+dd^*=1 \\
{\rm (D)} & a^*a+c^*c=1 & \quad & {\rm (E)} & ac^*+bd^*=0 & 
\quad & {\rm (F)} & a^*b+c^*d=0
\end{array}$$
and 
$${\rm (V)} \ \ \ vv^*=1=v^*v.$$

On the other hand the twisted determinant condition (TD) applied to the
permutations gives:
$$\begin{array}{llrlllr}
{\rm (a1)} &  {[123]} & E_{123}vad+E_{132}vbc=E_{123}1 
& \quad &
{\rm (a2)} &  {[132]} & E_{132}vda+E_{123}vcb=E_{132}1 
\\
{\rm (a3)} &  {[213]} & E_{213}avd+E_{312}bvc=E_{213}1
& \quad &
{\rm (a4)} &  {[231]} & E_{231}adv+E_{321}bcv=E_{231}1 
\\
{\rm (a5)} &  {[312]} & E_{312}dva+E_{213}cvb=E_{312}1 
& \quad &
{\rm (a6)} &  {[321]} & E_{321}dav+E_{231}cbv=E_{321}1 
\end{array}$$
and the other non-trivial relations are the following 
$$\begin{array}{llrlllr}
{\rm (a7)} &  {[122]} & E_{123}vab+E_{132}vba=0
& \quad &
{\rm (a8)} &  {[212]} & E_{213}avb+E_{312}bva=0 
\\
{\rm (a9)} &  {[221]} & E_{231}abv+E_{321}bav=0 
& \quad &
{\rm (a10)} &  {[133]} & E_{123}vcd+E_{132}vdc=0
\\
{\rm (a11)} &  {[313]} & E_{213}cvd+E_{312}dvc=0
& \quad &
{\rm (a12)} &  {[331]} & E_{231}cdv+E_{321}dcv=0.
\end{array}$$
The relations above together with ${\rm (V)}$ yields:
$$\begin{array}{rrllrrrll}
{\rm (a1)' :} &{\rm (1)} &  \Rightarrow  & v^*= ad+\frac{E_{132}}{E_{123}}\ bc
& \quad &
{\rm (a2)' :} & {\rm (2)} &  \Rightarrow & v^*= da+\frac{E_{123}}{E_{132}}cb \\
{\rm (a4)' :} & {\rm (4)} &  \Rightarrow & v^*=ad+\frac{E_{321}}{E_{231}}\ bc 
& \quad &
{\rm (a6)' :} & {\rm (6)} &  \Rightarrow & v^*=da+\frac{E_{231}}{E_{321}}cb \\
{\rm (a7)' :} & {\rm (7)} & \Rightarrow &  ab+\frac{E_{132}}{E_{123}}\ ba =0
& \quad &
{\rm (a9)' :} & {\rm (10)} & \Rightarrow &  ab+\frac{E_{321}}{E_{231}}\ ba=0\\
{\rm (a10)' :} & {\rm (10)} & \Rightarrow & cd+\frac{E_{132}}{E_{123}}\ dc=0
& \quad &
{\rm (a12)' :} & {\rm (12)} & \Rightarrow & cd+\frac{E_{321}}{E_{231}}\ dc=0.
\end{array}$$
So we have either
\begin{equation} \label{alpha2}
\mu:=\frac{E_{132}}{E_{123}}=\frac{E_{321}}{E_{231}}
\end{equation}
or all the following relations are satisfied: 
$$bc=cb=0, \quad ab=ba=0, \quad cd=dc=0, \quad v^*=ad.$$ 
The latter implies $b(ad)=0$ and $(ad)c=0$ and thus $b=0$ and $c=0$. In such
case the quantum group reduces to $u=v\oplus a \oplus d$, where all there
generators are commuting unitaries and $adv=1$. Then
$\G$ is isomorphic to the algebra of continuous functions on two-dimensional
torus.

So let us assume that \eqref{alpha2} holds. Then
\begin{eqnarray*}
{\rm (G) } & ad+\mu bc =& \!\!\!\! v^*=da+\frac{1}{\mu}cb  \\
{\rm (H) } & ab+\mu ba =0& ,\ \ cd+\mu dc=0.
\end{eqnarray*}

From (B), (D), (F), (G) and (H) we get 
\begin{eqnarray*}
d &=& a^* (ad)+ c^*(cd) = a^* (v^*-\mu bc ) + c^* (-\mu dc) 
= a^*v^* -\mu (a^*b+c^*d)c = a^*v^*,\\
a&=&  b^*(ba)+d^*(da)= 
b^*(-\frac{1}{\mu} ab) + d^* (v^*-\frac{1}{\mu}cb) 
=d^* v^* - \frac{1}{\mu} (b^*a+d^*c)b=d^*v^*,
\end{eqnarray*}
which implies 
\begin{equation} \label{eq_uq2_d}
 d=a^*v^*, \quad av=va.
\end{equation}

Similarly, $b = a^* (ab)+ c^*(cb) = \mu c^*v^* -\mu (a^*b+c^*d)a = \mu
c^*v^*$ and $c = b^*(bc)+d^*(dc)= \frac{1}{\mu} b^*(v^*-ad) - d^*
(\frac{1}{\mu}cd) =\frac{1}{\mu} b^*v^*$,
which yields 
\begin{equation} \label{eq_uq2_b}
 b=\mu c^*v^*= \bar{\mu} v^*c^*, \quad cv=e^{-2\I t}vc,
\end{equation}
when $\mu=|\mu|e^{\I t}$. 

Inserting $ d=a^*v^*$ and $ b=\mu c^*v^*=\bar{\mu} v^*c^*$ into (A)-(C):
\begin{eqnarray*} 
{\rm (A)} & 1=& aa^*+bb^*=aa^* + \mu c^*v^*(\mu c^*v^*)^* = aa^* +|\mu|^2 c^*c
\\
{\rm (B)} & 1=& (\bar{\mu} v^*c^*)^*\bar{\mu} v^*c^*+(a^*v^*)^*a^*v^*=|\mu|^2
cc^*+aa^* \\
{\rm (C)} & 1=& cc^*+dd^*=cc^* + a^*v^*(a^*v^*)^* = cc^* +a^*a,
\end{eqnarray*}
and comparing with (D): $1= a^*a+c^*c$, we see that $c$ is normal ($c^*c=cc^*$).
Moreover, 
\begin{eqnarray*}
{\rm (E)':} & 0= & ac^*+bd^*=ac^*+ \mu c^*v^*(a^*v^*)^* =ac^*+\mu c^*a\\
{\rm (F)':} & 0= & a^*b+c^*d=\mu a^*c^*v^*+c^*a^*v^* \quad \Rightarrow \quad 
ac+\bar{\mu} ca=0.
\end{eqnarray*}
Therefore we conclude that $a,c,v$ are the generators of $\A$ satisfying 
\begin{equation} \label{eq_r1}
\tag{R1} ac+\bar{\mu} ca=0, \quad ac^*+\mu c^*a=0, \quad a^*a+c^*c=1, \quad aa^*
+|\mu|^2 c^*c=1, \quad cc^*=c^*c
\end{equation}
and 
\begin{equation} \label{eq_r2}
\tag{R2}  vv^*=v^*v=1, \quad av=va, \quad cv=e^{-2\I t}vc.
\end{equation}

Finally, to check when such elements (with $b=\mu c^*v^*$ and $d=a^*v^*$)
satisfy both the unitarity and the twisted determinant conditions, it is enough
to look at those relations among (a1)-(a12) which were not used yet. These are
(a3), (a5), (a8) and (a11), which yield: 
$$\begin{array}{ll}
{\rm (a3):} &  1= avd+\frac{E_{312}}{E_{213}}bvc = av a^*v^*+
\frac{E_{312}}{E_{213}}\mu c^*v^*vc \quad
\Rightarrow \quad aa^*+ \frac{E_{312}}{E_{213}}\mu c^*c =1 ;
\\
{\rm (a5):} &  1= dva+\frac{E_{213}}{E_{312}}cvb = 
a^*v^*va+\frac{E_{213}}{E_{312}}\ cv\bar{\mu} v^*c^* \quad \Rightarrow \quad
a^*a+\frac{E_{213}}{E_{312}} \bar{\mu}\ cc^*=1;
\\
{\rm (a8):} &  0= avb+\frac{E_{312}}{E_{213}} bva=
\bar{\mu} avv^*c^*+\frac{E_{312}}{E_{213}}\mu c^*v^*va \quad
\Rightarrow \quad ac^*+\frac{E_{312}}{E_{213}}\frac{\mu}{\bar{\mu}} c^*a =0; 
\\
{\rm (a11):} & 0= cvd+ \frac{E_{312}}{E_{213}}dvc=
cvv^*a^*+\frac{E_{312}}{E_{213}} a^*v^*vc \quad \Rightarrow
\quad ca^*+\frac{E_{312}}{E_{213}}a^*c=0.
\end{array}$$
They are compatible with \eqref{eq_r1} and \eqref{eq_r2} if and only if 
\begin{equation} \label{alpha3}
\frac{E_{312}}{E_{213}}=\bar{\mu}=
\frac{\bar{E}_{321}}{\bar{E}_{231}}.
\end{equation}

If \eqref{alpha3} does not hold, then 
$$aa^*+ \frac{E_{312}}{E_{213}}\mu c^*c =1 = aa^* +|\mu|^2 c^*c$$
implies that $c^*c=0$ and thus $c=0$. Since, moreover, $b=\mu c^*v^*=0$, we are
back to the case of two-dimensional torus. 

So it remains to consider the case when \eqref{alpha2} and
\eqref{alpha3} are satisfied (and $E_{123}=1$). Then 
the function $E$ must be of the form
\begin{eqnarray*} 
&E_{123}=1, \; E_{132}=\mu, \; E_{213}=\alpha, & \\ 
&E_{231}=\beta, \; E_{312}=\alpha\bar{\mu}, \; E_{321}=\beta\mu &
\end{eqnarray*}
for some (non-zero) complex constants $\alpha$, $\beta$ and $\mu$. 
The C${}^*$-algebra of the related Woronowicz quantum group is generated
by $a,c,v$ satisfying relations \eqref{eq_r1} and \eqref{eq_r2}.
But then all relations in (U) and (TD) hold. 
This means that the quantum group under consideration is exactly the compact
quantum group $U_q(2)$ for $q=-\mu\in \C^*$ (see Definition \ref{def_uq2}). 

Returning to the general case, let us assume that the
matrix $u$ decomposes in such a way that $u_{kk}$ is the unique non-zero element
in the $k$-th column and the $k$th row and $r$ is the first row of the $2\times
2$-block. Then consider 
$\tilde{u}=[u_{\sigma_0(i),\sigma_0(j)}]_{i,j=1}^3$, where
$\sigma_0:(123)\mapsto (kir)$. Then $\tilde{u}$ is of the form
\eqref{eq_u_decomp_form} and, by Lemma \ref{lem_iso}, the two quantum groups are
isomorphic. Translating \eqref{eq_r1} and \eqref{eq_r2} to this general case, we
recover the general conditions for a PW-quantum group with the two-block
decomposition to be isomorphic to $U_q(2)$. 
\end{proof}

\begin{rem} \label{rem_uq2_and_morphisms}
Note that the relations {\rm (G)}, \eqref{eq_uq2_d} and \eqref{eq_uq2_b} can
be obtained from \eqref{eq_R_star}, whereas {\rm (H)} and the normality of $c$
follows from \eqref{eq_Q_row}, \eqref{eq_Q_col} and \eqref{eq_Q_mix12} after
substitution $b=\bar{\mu}v^*c^*$ and $d=v^*a^*$. Moreover, the
condition \eqref{alpha2} means that $\co{3}{2}=1$. 
\end{rem}

\begin{rem} \label{rem_2block_decomp}
When the fundamental corepresentation of a PW-quantum group
admits a two-block decomposition and the resulting quantum group is $U_q(2)$,
then two of the diagonal constants, defined by \eqref{eq_const_diagonal},
are equal. For the case with 1-dimensional block for $k=1$ we get $p_2=p_3$,
i.e. 
\begin{equation*} 
\widebar{E_{213}}E_{132}+\widebar{E_{231}}E_{312}=\widebar{E_{312}}E_{
123}+\widebar { E_ { 321}}E_{ 213}.
\end{equation*}
Conversely, if we know that $p_1\neq p_2$ and $p_2=p_3$, then the
fundamental corepresentation of a PW-quantum group admits a two-block
decomposition and Theorem \ref{thm_uq2_or_t2} describes all possible PW-quantum
groups.  
\end{rem}

\section{Classification result}
\label{sec_class}
We devote this Section to the proof of the Classification Theorem. 

Theorem \ref{thm_comm_rel} reveals that commutation relations among
generators depend on the characteristic constants $\co{n}{l}$ ($n\neq l$). On
the other hand, Equation \eqref{eq_char_const_rel} suggests that one of the
following cases must occur:
\begin{enumerate}
\item[{\bf (1)}] none of the constants equals 1; this case will be considered in
Section \ref{sec_none}. 
\item[{\bf (2)}] exactly two of the constants equal 1, in Section
 \ref{sec_some}. 
\item[{\bf (3)}] all the six constants equal 1, in Section \ref{sec_special}. 
\end{enumerate}

\subsection{Case (1): all characteristic constants different from 1.}
\label{sec_none}

\begin{prop}\label{prop_none_kac_cases}
 Let $\G_a$ be a PW-quantum group such that $\co{j}{n}\neq 1$ for any $j\neq
n$, and let $M_n= \sum_{jk} |E_{njk}|^2$ $(n=1,2,3)$ be the modular constants as
in \eqref{eq_const_modular}. Then exactly one of the following happens:
 \begin{enumerate}
  \item$\G_a \cong A_{p,k,m}(3)$ if $E$ is as in Definition
\ref{def_akm},
  \item $\G_a \cong U_q(2)$ $($when $M_j=M_l \neq M_n$,
$\frac{E_{nlj}}{E_{njl}} =\frac{E_{ljn}}{E_{jln}}$ and $\frac{E_{nlj}}{E_{nlj}}
= \frac{\bar{E}_{lnj}}{\bar{E}_{jnj}})$,
  \item $\G_a \cong C(\T^2)$ $($otherwise$)$.
 \end{enumerate}
\end{prop}

\begin{proof}
With the assumption $\co{j}{n}\neq 1$ for any $j\neq n$, we use Lemma
\ref{lem_akm_relations} to deduce that the generators of the quantum group
$\G_a$ satisfy \eqref{eq_s3_col}-\eqref{eq_s3_mix}. In particular, every
$u_{ij}$ is a normal partial isometry (cf.\ Theorem \ref{thm_pi}). Moreover, by
\eqref{eq_modular_prop}, we know that
$$\sum_{s=1}^3 M_s u_{sa}u^*_{sb} =\delta_{ab} M_a \one, \quad 
\sum_{s} \frac{1}{M_s} u^*_{as}u_{bs}=\frac{1}{M_a}\delta_{ab} \one. $$
Multiplying both sides of the first equation (with $a=b$) by $u_{rb}$ and using
$u^*_{sb}u_{rb} =0$ for $r\neq s$ and $u_{rb} u^*_{rb}u_{rb} = u_{rb}$, we find
out that
\begin{equation*}
 M_{r}u_{rb} = M_{b} u_{rb}.
\end{equation*}
So for $r\neq b$ either $M_r=M_b$ or $u_{rb}=0$. 

When the constants $M_1$, $M_2$ and $M_3$ are pairwise different, then only
diagonal terms of the matrix $u$ can be non-zero, i.e.\ $u=u_{11}\oplus
u_{22}\oplus u_{33}$ and $\G_a \cong \T^2$.

When exactly two of $N$'s are equal, then the fundamental
corepresentation $u$ decomposes into $u=v\oplus w$, where $w$ is the 2 by 2
matrix, $v=u_{nn}$. Due to Theorem \ref{thm_uq2_or_t2}, $\G_a$ is then
isomorphic to $U_q(2)$ iff $\frac{E_{nlj}}{E_{njl}} = \frac{E_{ljn}}{E_{jln}}$
and $\frac{E_{nlj}}{E_{nlj}} = \frac{\bar{E}_{lnj}}{\bar{E}_{jnj}}$. 
Otherwise again, $u=u_{11}\oplus u_{22}\oplus u_{33}$ and $\G_a \cong \T^2$. 

When all the three constants $M_1$, $M_2$ and $M_3$ are equal (and necessarily
non-zero), we get:
$$ \sum_{r=1}^3 u_{rn} u^*_{rn} = 1, \quad 
 \sum_{n=1}^3 u^*_{rn} u_{rn} = 1.$$
These relations imply that $u^t$ is unitary and thus $\G_a$ is of Kac type.
Therefore we can apply Lemma \ref{lem_akm_invariance} to ensure that whenever
$\pi$ is an irreducible representation of $\G_a$ on a Hilbert space $H$, then
$K_{ij}=\pi(u_{ij})(H)$ must be either $\{0\}$ or $H$. Following the proof of
Theorem \ref{thm_akm_reps}, one can show that $\pi$ must be of the form 
$$ \pi(u_{1,\sigma(1)})=V_1, \quad \pi(u_{2,\sigma(2)})=V_2, \quad
\pi(u_{3,\sigma(3)})=V_3$$
for some permutation $\sigma\in S_3$, and $\pi(u_{ij})=0$ otherwise.
Moreover, the permutation $\sigma$ has to fullfil the condition
\eqref{eq_akm_perm_cond}. Evidently, it cannot be a transposition exchanging
$r$ with $k$ due to $\co{r}{k}\neq 1$. When the array $E$ is as in Definition
\ref{def_akm}, then the three permutation: $\id$, $(231)$ and $(312)$ are
allowed. This leads to the quantum group $A_{p,k,m}(3)$. If the array is not of
that form, only $\sigma=\id$ is permitted. Then the fundamental corepresentation
$u$ decomposes under any irreducible representation $\pi$ into a direct sum
$\pi(u) = V_1\oplus V_2\oplus V_3$ of commuting unitaries. In the same way as in
Theorem \ref{thm_akm_reps}, we conclude that $\G_a \cong C(\T^2)$.
\end{proof}

\subsection{Case (2): two characteristic constants equal 1.} 
\label{sec_some}

The next step is to consider the situation when exactly two of the constants
equals 1, i.e.\ $\co{l}{n}=\co{n}{l}=1$ for some $n\neq l$. In view of Lemma
\ref{lem_iso}, without loose of generality we can assume that $\co{2}{3}=1$,
$\co{1}{3}\neq 1$.  The next three results describe this situation in a more 
detailed way. Namely, they first reveal additional relations between the
generators and, next, their influence on the $*$-representations of the algebra
of such a quantum group. This exhibits how such a quantum group 
looks like. 

\begin{lemma} \label{lem_supp_cond_12_21}
 Let $\G_b$ be a PW-quantum group related to an array $E$ which satisfies
$\co{2}{3}=1$ and $\co{1}{3}\neq 1$. Then  
\begin{eqnarray} 
\label{eq_some_row_col_0} & & 
u_{1a} u_{na}= u_{na} u_{1a} =0, \; u_{an} u_{a1}= u_{a1}u_{an}=0, \quad n\neq
1, \; a=1,2,3,
\\ \label{eq_some_with_star} & & 
u_{1n}u_{1l}^* =0, \quad u_{1n}^*u_{1l} =0, \quad n\neq l,
\\ \label{eq_some_antidiag} & & 
u_{21}u_{12}=0=u_{12}u_{21}, \quad u_{31}u_{13}=0=u_{13}u_{31}. 
\end{eqnarray}
Moreover, $u_{1n}$ and $u_{n1}$ are partial isometries, and $u_{11}$ is normal.
\end{lemma}

\begin{proof}
When $\co{1}{3}\neq 1$ and $\co{2}{3}=1$, then $\co{3}{1}=\frac1{\co{1}{3}}\neq
1$ and
$\co{1}{2}=\co{1}{3}\co{3}{2}=\frac{\co{1}{3}}{\co{2}{3}}=\co{1}{3}\neq 1$.
Thus, by \eqref{eq_comm_rel_row0}, the product of elements in the same row (in
the same column, respectively) with one of them being in the first row (resp.\
column) vanishes, so that Equation \eqref{eq_some_row_col_0} holds. 

Due to \eqref{eq_some_row_col_0} and Theorem \ref{thm_pi}, we conclude that
$u_{11}$ is a normal partial isometry. In fact, repeating the same reasoning
as at the beginning of the proof of Theorem \ref{thm_pi}, we can show that 
\begin{equation*} 
u_{1k}u_{1r}^* =0, \quad u_{1k}^*u_{1r} =0 \quad \mbox{for any} \; k\neq r.
\end{equation*}
and that the generators $u_{1n}$ $(n=1,2,3)$ are partial isometries. Application
of the antipode shows that $u_{r1}$ $(r=1,2,3)$ are partial isometries too. 


To see that \eqref{eq_some_antidiag} holds, let us fix $r=l\in \{2,3\}$ and
$n=k=1$. Then $\co{l}{n} \neq 1$ and thus
\eqref{eq_comm_rel_not1} describes the
relation between $u_{rk}u_{kr}$ and $u_{kr}u_{rk}$. More precisely, we have
$$u_{rk} u_{kr} = \frac{E_{irk}}{E_{ikr}}\frac{E_{irk}}{E_{ikr}} 
\frac{1 -\co{k}{r}}{1- \co{r}{k} }\,  u_{kr} u_{rk}
+ \frac{E_{irk}}{E_{ikr}} \frac{1-\co{r}{k}\co{k}{r}}{1- \co{r}{k}}\, 
u_{kk} u_{rk} = -\co{k}{r} \left(\frac{E_{irk}}{E_{ikr}}\right)^2 u_{kr} u_{rk}.
$$
On the other hand, we can compare two different ways of expressing $u_{ii}^*$
$$ \frac{E_{irk}}{E_{irk}} u_{rr} u_{kk}+ \frac{E_{ikr}}{E_{irk}}
u_{rk}u_{kr} = u^*_{ii} = 
\frac{\bar{E}_{irk}}{\bar{E}_{irk}} u_{rr} u_{kk}+
\frac{\bar{E}_{ikr}}{\bar{E}_{irk}} u_{kr}u_{rk}, $$
(cf.\ Equations \eqref{eq_R_star}, \eqref{eq_R_star_adjoint}) and then 
observe that 
$$  u_{rk}u_{kr} = \frac{E_{irk}}{E_{ikr}}
\frac{\bar{E}_{ikr}}{\bar{E}_{irk}} u_{kr}u_{rk}. $$
Together with the previous relation this yields
$$ u_{kr} u_{rk} = 
-\frac{E_{rki}}{E_{kri}} \frac{\bar{E}_{irk}} {\bar{E}_{ikr}}u_{kr}u_{rk},  $$
from which we deduce that (for $k=1$)
\begin{equation} \label{eq_some_antidiag_condition}
\frac{E_{rki}}{E_{kri}} \frac{\bar{E}_{irk}}
{\bar{E}_{ikr}}\neq -1 \quad \mbox{implies} \quad u_{kr} u_{rk} =
0=u_{rk}u_{kr}. 
\end{equation}

Now it is a matter of algebraic computation to see that $\frac{E_{rki}}{E_{kri}}
\frac{\bar{E}_{irk}} {\bar{E}_{ikr}}\neq -1$ for $r\in \{2,3\}$ and $k=1$.
Indeed, let us assume that the array $E$ satisfies:
\begin{equation} \label{eq_set_of_cond}
\co{2}{3}=1, \quad \alpha:= \co{1}{3} \neq 1, \quad p_1=p_2=p_3
\end{equation}
and the supplementary condition for $r\in \{2,3\}$ and $k=1$ 
\begin{equation} \label{eq_some_suppl_rk}
\frac{E_{rki}}{E_{kri}} \frac{\bar{E}_{irk}} {\bar{E}_{ikr}}=-1. 
\end{equation}
Then, with the notation from Remark \ref{rem_notation}, the relations of
\eqref{eq_set_of_cond} can be written as 
$$ t=pr, \quad s=\frac{\alpha p r}{q}, \quad (\#) \ \left\{ \begin{array}{l}
r+ \bar{p}t = p\bar{q}+\bar{r}s \\
r+ \bar{p}t = \bar{s} + q\bar{t}
\end{array} \right. .$$

Let us first assume that \eqref{eq_some_suppl_rk} holds for $r=2$, i.e.\ 
$s=-p\bar{q}r$. Replacing $s$ and $t$ in the second relation of $(\#)$ we get
$r(1+ |p|^2) = -\bar{p}q\bar{r} + q\bar{p}\bar{r} =0,$
contrary to our assumption $r\neq 0$. Thus
\eqref{eq_some_suppl_rk} cannot hold for $r=2$. 

Similarly, if we assume that \eqref{eq_some_suppl_rk} holds for $r=3$,
then $s = -\frac{p\bar{q}}{\bar{r}}$ and the first relation in $(\#)$ yields
$r(1+ |p|^2) = p\bar{q}+\bar{r}s = 0, $
which again contradicts the fact that $r\neq 0$. Hence \eqref{eq_some_suppl_rk}
cannot hold and \eqref{eq_some_antidiag_condition} forces
\eqref{eq_some_antidiag} to hold.
\end{proof}


\begin{lemma} \label{lem_11_inv} 
Let $\G_b$ be a PW-quantum group related to an array $E$ which satisfies
$\co{2}{3}=1$ and $\co{1}{3}\neq 1$ and let $\pi$ be a $*$-representation of
$\G_b$ on a Hilbert space $H$. Then 
\begin{enumerate}
 \item the subspace $\pi(u_{11})(H)$ is $\pi$-invariant. 
 \item 
 if $\pi(u_{11})=0$ then the subspace $\pi(u_{12}u_{12}^*)$ is $\pi$-invariant.
 \end{enumerate}
\end{lemma}

\begin{proof}
We already know that $u_{11}$ is a normal partial isometry, so
$K:=\pi(u_{11})(H)=\pi(u_{11}u_{11}^*)(H)$. 
Moreover, when $r\neq 1$ and $n\neq 1$, and $k$ and $l$ are the unique integers
such that $(1,r,k), (1,n,j)\in S_3$, then $\co{r}{k}=\co{j}{n}= 1$. Hence 
we can apply Lemma \ref{lem_ij_rn_star}, to show that 
$u_{rn}u_{11}^* = A_{r,k}^{n,j} u_{11}^* u_{rn},$
where $A_{r,k}^{n,j}\neq 0$. 
The application of the relations between generators and their adjoints
(Corollary \ref{cor_comm_rel} and Lemma \ref{lem_supp_cond_12_21}) yields 
 $$u_{rn}(u_{11}u_{11}^*) = \left\{ 
 \begin{array}{ll}
(u_{11}u_{11}^*)u_{11} & \mbox{if } r=n=1  \quad (\mbox{by normality}),\\
 u_{1n}u_{11}u_{11}^* =0 = (u_{11}u_{11}^*)u_{1n} & \mbox{if } r=1, n\neq 1, \\
 u_{r1}u_{11}u_{11}^* =0 = (u_{11}u_{11}^*)u_{r1} &  \mbox{if } r\neq 1, n=1,\\
u_{rn}u_{11}^*u_{11}= 
A_{r,k}^{n,j}A_{r,1}^{n,1}(u_{11}u_{11}^*)u_{rn},
 &\mbox{if } r\neq 1,n\neq 1,
 \end{array} \right.$$
and
$$u_{rn}^*(u_{11}u_{11}^*) = \left\{ 
 \begin{array}{ll}
(u_{11}u_{11}^*)u_{11}^* & \mbox{if } r=n=1,\\
u_{1n}^* u_{11}u_{11}^*=0 =
(u_{11}^*u_{11})u_{1n}^* = (u_{11}u_{11}^*)u_{1n}^*
 & \mbox{if } r=1, n\neq 1,\\
u_{r1}^*(u_{11}u_{11}^*)=0 = (u_{11}u_{11}^*)u_{r1}^* &  \mbox{if } r\neq 1,
n=1,\\
u_{rn}^*u_{11}u_{11}^* 
= (\overline{A_{r,k}^{n,j}})^{-1} \overline{A_{1,r}^{1,n}}
(u_{11}u_{11}^*) u_{rn}^*
 &\mbox{if } r\neq 1,n\neq 1. 
 \end{array} \right.$$
Thus whenever $x\in \pi(u_{11}u_{11}^*)(H)$, then $\pi(u_{rn})x, \pi(u_{rn}^*)x
\in \pi(u_{11}u_{11}^*)(H)$. 


As for the second part of the proof, let us assume that $\pi(u_{11})=0$. 
Let us denote $U_{ij}=\pi(u_{ij})$. It is enough to show that $U_{ij}^{\epsilon}
(U_{12}U_{12}^*) = \mbox{const}\cdot (U_{12}U_{12}^*) U_{ij}^{\epsilon}$ for any
$i,j \in \{1,2,3\}$ and $\epsilon\in \{1,*\}$. Indeed, if
$x=(U_{12}U_{12}^*) y\in U_{12}U_{12}^*(H)$, then 
$U_{ij}^{\epsilon} x = U_{ij}^{\epsilon} (U_{12}U_{12}^*) y = \mbox{const}\cdot
(U_{12}U_{12}^*) U_{ij}^{\epsilon} y \in U_{12}U_{12}^*(H)$. The constant can
be 0. 

Due to Lemma \ref{lem_supp_cond_12_21}, we know that
$U_{21}U_{12}=0=U_{12}U_{21}$ and $U_{31}U_{13}=0=U_{13}U_{31}$, and thus, by
\eqref{eq_R_star}, 
$$ U_{22}^* = \frac{E_{213}}{E_{213}} U_{11}U_{33} + \frac{E_{231}}{E_{213}}
U_{13}U_{31}=0, \quad 
U_{33}^* = \frac{E_{312}}{E_{312}} U_{11}U_{22} + \frac{E_{321}}{E_{312}}
U_{12}U_{21}=0.$$
Hence $ U_{kk}(U_{12}U_{12}^*)=0$ and $U_{kk}^*(U_{12}U_{12}^*)=0$ for
$k=1,2,3$. Of course, $U_{21}U_{12}U_{12}^*=0$ and, due to
\eqref{eq_some_row_col_0}, $U_{32}U_{12}U_{12}^* = 0$. Moreover, by
\eqref{eq_some_with_star}, we have $U_{13}^*U_{12} = 0$, which also implies
$$ 
U_{13}U_{12}U_{12}^* = -\frac{\bar{E}_{123}}{\bar{E}_{132}}
U_{12}U_{13}U_{12}^*=0.$$
Now, using \eqref{eq_R_star} again, we check that 
\begin{equation*}
U_{23}U_{12} = \frac{E_{321}}{E_{132}} U_{31}^* =
\frac{E_{321}}{E_{312}}\frac{E_{123}}{E_{132}}U_{12}U_{23}, \quad 
U_{31}U_{12} = \frac{E_{231}}{E_{312}} U_{23}^* =
\frac{E_{231}}{E_{213}}\frac{E_{321}}{E_{312}}U_{12}U_{31}, 
\end{equation*}
and (in the same way)
$$
U_{21}^*U_{12} = 0, \quad U_{31}^*U_{12} =0, \quad 
U_{23}^*U_{12} = \frac{E_{321}}{E_{312}}\frac{E_{231}}{E_{213}} 
U_{12} U_{23}^* ,  \quad 
U_{32}^*U_{12} = \frac{E_{321}}{E_{312}}\frac{E_{231}}{E_{213}} 
U_{12}U_{32}^*.
$$
Combining these relations we find out that $U_{ij}^{\epsilon}
(U_{12}U_{12}^*) = \mbox{const}\cdot (U_{12}U_{12}^*) U_{ij}^{\epsilon}$ for
$i\neq1$ and $j\neq 2$. To finish the proof it remains to check that
$U_{12}^*(U_{12}U_{12}^*)= \mbox{const}\cdot (U_{12}U_{12}^*)U_{12}^*$. But
this follows from 
\begin{eqnarray*}
U_{12}^*U_{12} &=& \frac{E_{231}}{E_{132}} U_{33}
U_{21} U_{12} + \frac{E_{213}}{E_{132}} U_{31}U_{23}U_{12}
= \frac{E_{231}}{E_{123}} U_{31}
\left(\frac{E_{321}}{E_{312}}\frac{E_{123}}{E_{132}}U_{12}U_{23} \right)
\\
&=& 
\frac{E_{231}}{E_{123}} 
\frac{E_{321}}{E_{312}}\frac{E_{123}}{E_{132}}
\left(\frac{E_{231}}{E_{213}}\frac{E_{321}}{E_{312}} U_{12}U_{31}\right) U_{23}
= \left(\frac{E_{321}}{E_{312}}\right)^2 \frac{E_{123}}{E_{132}}
\frac{E_{231}}{E_{213}} U_{12}U_{12}^*.
\end{eqnarray*}
\end{proof}

We are ready to describe all PW-quantum groups with a single pair of
characteristic constants being 1. 
\begin{prop} \label{prop_some}
 The PW-quantum group $\G_b$ related to an array $E$ which satisfies
$\co{2}{3}=1$ and $\co{1}{3}\neq 1$ is isomorphic to $U_q(2)$ or $C(\T^2)$. 
\end{prop}

\begin{proof}
Let $\pi$ be an irreducible representation of the algebra $\A=C(\G_b)$ on a
Hilbert space $H$. According to Lemma \ref{lem_11_inv},
$K:=\pi(u_{11})(H)$ is invariant
and thus (due to irreducibility) must be trivial:
$K=\{0\}$ or $K=H$. If $K=H$ then $\pi(u_{11})$ must be a unitary and the
relation (U) implies that the other terms in the first row and in the first
column vanish: $\pi(u_{12})=\pi(u_{13})=\pi(u_{21})=\pi(u_{31})=0$. 

So let us assume the contrary, that is $K=\{0\}$ or, equivalently,
$\pi(u_{11})=0$. In view of Lemma \ref{lem_11_inv}, part (2), the subspace $L=
\pi(u_{12}u_{12}^*)(H)$ is invariant, hence trivial. If $L=H$, then, by the same
reasoning as above, $\pi(u_{12})$ is unitary and 
$\pi(u_{11})=\pi(u_{13})=\pi(u_{22})=\pi(u_{32})=0$. By
\eqref{eq_some_antidiag}, $\pi(u_{21})=0$ and thus $\pi(u_{23})$ and
$\pi(u_{31})$ must be unitaries. Thus $V_1=\pi(u_{12})$,
$V_2=\pi(u_{23})$ and $V_3=\pi(u_{31})$ are unitaries, and they are the only
non-zero elements in $\pi(u)$. In this case, with the notation
as in Remark \ref{rem_notation}, the compatibility of the twisted determinant
condition (cf.\ \eqref{eq_akm_perm_cond} in the proof of Theorem
\ref{thm_akm_reps}), implies 
that $t=\frac{p^2}{q}$ and $s=r^2$. But then 
$1=\co{2}{3}=\frac{t}{pr}=\frac{p}{qr}$ yields
$\co{1}{3}=\frac{sq}{rp}=\frac{r q}{p}=1$, contrary to the assumption.
Thus we cannot have $L=\pi(u_{12}u_{12}^*)(H)=H$.

Similarly, if $L=\{0\}$ then $\pi(u_{12})=0$ and by the
unitarity relation $\pi(u_{13})$ must be unitary, $\pi(u_{23})$ and
$\pi(u_{33})$ vanish, and by \eqref{eq_some_antidiag}, $\pi(u_{31})$ vanishes
too. Thus $W_1=\pi(u_{13})$, $W_2=\pi(u_{21})$ and $W_3=\pi(u_{32})$ are
unitaries, and they are the only non-zero elements in $\pi(u)$. Again, the
compatibility conditions induced by (TD) contradicts the assumption that
$\co{2}{3}=1$ and $\co{1}{2}\neq 1$. Thus $L=\pi(u_{12}u_{12}^*)(H)$ cannot be
trivial. 

This means that $K=\pi(u_{11}u_{11}^*)(H)=H$, and under every
irreducible representation $\pi$ of $\G$ the matrix $u$ decomposes into two
blocks. Thus, according to Theorem \ref{thm_uq2_or_t2}, $\G\cong U_q(2)$ or
$\G\cong \T^2$. 
\end{proof}

\subsection{Case (3): all characteristic constants equal to 1.}
\label{sec_special}
We are left with the situation when all characteristic
constants equal 1, i.e.\ $\co{l}{n}=1$ for any $n\neq l$. In view of
\eqref{eq_char_const_rel}, it is equivalent to the assumption that $\co{1}{2} =
\co{1}{3}=1. $
Moreover, according to Remark \ref{rem_2block_decomp}, we can restrict our
consideration to the case, where all values on the diagonal of the intertwiner
$P$ are equal, i.e.\ $p_1=p_2=p_3$. 

We first show that these algebraic
conditions on the array $E$ restrict the possible solutions to four cases,
related to four sets of parameters associated to $E_{rki}$, $(r,k,i)\in S_3$.
Next we prove that they describe only two non-isomorphic quantum groups, the
ones defined in Subsections \ref{ssec_ex_s_k} and \ref{ssec_ex_a}. 

\begin{lemma} \label{lem_4_solutions}
If a normalized array $E$ related to a PW-quantum group satisfy 
 \begin{equation} \label{eq_4_solutions}
 p_1=p_2=p_3 \quad \mbox{and} \quad \co{1}{2} = \co{1}{3}=1,   
 \end{equation}
it must take one of the following form: 
\begin{equation*}
\begin{array}{c||c|c|c|c|c|c}
 & E_{123} & E_{132}& E_{213} & E_{231}
& E_{312} & E_{321} \\
 \hline  \hline 
{\rm Case} \, 1 & 
1 & p & p\zeta^{-m} 
& \zeta^m & \zeta^{-m} & p\zeta^m \\  \hline 
{\rm Case} \, 2 & 
1& p & p \zeta^{-m} & |p|^2 \zeta^m & |p|^2 \zeta^{-m} & |p|^2 p \zeta^{m} \\ 
\hline 
{\rm Case} \, 3 & 
1 & p & \frac1{\bar{p}} \zeta^{-m} & \zeta^m & |p|^2 \zeta^{-m} & p\zeta^m \\
\hline 
{\rm Case} \, 4 & 
1 & p & \frac1{\bar{p}}\zeta^{-m} & \frac1{|p|^2} \zeta^m & \zeta^{-m} &
\frac1{\bar{p}} \zeta^m
\end{array}
\end{equation*}
where $p\in \C^*$ and $\zeta=e^{\frac23 \i\pi}$, $m\in \{0,1,2\}$.
\end{lemma}

\begin{proof}
If we use the notation from Remark \ref{rem_notation}, then
\eqref{eq_4_solutions} rewrites as
\begin{equation*} 
t=pr, \qquad s=\frac{pr}{q} \quad \mbox{and} \quad 
r+\bar{p}t= \bar{q}p+ \bar{r}s= \bar{s}+ \bar{t}q.
\end{equation*}
Substitution of $t$ and $s$ yields
\begin{equation} \label{eq_all_relations}
\frac{1+|p|^2}{ p}= \frac{|q|^2+ |r|^2}{rq} 
\quad \mbox{and}\quad  
\frac{|q|^2+ |r|^2}{rq} = \frac{r}{\bar{r}} \frac{1+ |q|^2}{q},
\end{equation}
and the second relation solved with respect to $|q|^2$ gives 
$$|q|^2 (r^2-\bar{r}) = \bar{r}|r|^2-r^2.$$
When $r^2=\bar{r}$, then necessarily $|r|=1$ ($r\neq 0$) and the relation is
satisfied. In fact, one easily checks that 
$|r|=1$ if and only if $r^2=\bar{r}$. On the other hand, if $|r|\neq 1$, then 
\begin{equation*} 
|q|^2 = \frac{\bar{r}|r|^2-r^2}{r^2-\bar{r}}. 
\end{equation*}
Checking when, for a complex $r$, one gets a real (and positive) term on the
right-hand side of the expression above, we find out that it must take the form
$r=2a \zeta^m$ with $m\in \{0,1,2\}$, $a>0$ and $a\neq \frac12$.
Then $|q|^2 = 2a=|r|$. 
Moreover, taking the absolute value of the first relation in
\eqref{eq_all_relations}, we
find out that in both cases ($|r|=1$ and $|r|=|q|^2$) we have
$|q|(1+|p|^2)= |p|(1+ |q|^2)$.
Therefore $|p||q|^2-(1+|p|^2)|q|+ |p|=0 $, and so $|p|=|q|$ or
$|p|=\frac1{|q|}$.
Finally, setting $p=|p| e^{\i \phi}$, $q=|q|e^{\i \xi}$, $r=|r|e^{\i \theta}$ we
find out from \eqref{eq_all_relations} that $e^{\i \xi}= e^{\i (\phi-\theta)}$.

Summing up we have the following solutions to the set of equations
\eqref{eq_4_solutions}:
Case 1 for $|p|=|q|$ and $|r|=1$, Case 2 when $|p|=|q|$ and $|r|=|q|^2$, Case
3 if $|p|=\frac1{|q|}$ and $|r|=1$, and Case 4 when $|p|=\frac1{|q|}$ and
$|r|=|q|^2$.
\end{proof}

Let us denote by $\G_k$ ($k=1,2,3,4$) the Woronowicz quantum group related to
the constants given by the Case $k$ from the table above. Note that 
$\G_1=A_{p,-m,m}(3)$ and $\G_2=SU_{p,m}(3)$, cf.\ Definitions \ref{def_akm} and
\ref{def_su_k}. The next result shows that these two families are all what we
can get from Woronowicz construction (under the permutation condition) when all
$\co{n}{l}=1$ for $n\neq l$. 

\begin{prop} \label{prop_su_isom}
 We have the following isomorphisms: 
 $$ \G_3 \cong SU_{\bar{p}\zeta^{-m},m}(3), \qquad \G_4
\cong SU_{\frac1{p},m}(3).$$
\end{prop}

\begin{proof}
For $\G_3$ take $\sigma=(213)$ and $c=\bar{p}\zeta^{-m}$. For
$\G_4$ take $\sigma=(132)$ and $c=\frac1{p}$. Then check that Lemma
\ref{lem_iso} is satisfied with the these permutations and constants. 
\end{proof}

\subsection{Proof of Classification Theorem}
 We sum up the results described throughout this Section to provide a complete
proof of the Theorem \ref{thm_classification}. Let $\G$ be a PW-quantum group
related to a (normalized) array $E$. 

If the diagonal constants $p_j$'s defined
by \eqref{eq_const_diagonal}, are pairwise different, then according to
Theorem \ref{thm_decomp} $u=u_{11}\oplus u_{22}\oplus u_{33}$. By the
unitarity condition (U), $u_{jj}$ is unitary and the twisted determinant
condition (TD) shows that the unitaries commute. Moreover,
$u_{33}=u_{11}^*u_{22}^*$. This means that $\G\cong C(\T^2)$. 

If $p_i=p_r\neq p_k$, then - again by Theorem \ref{thm_decomp} - $u$ decomposes
into two blocks: $1\times 1$ and $2\times 2$. In such a case we can apply
Theorem \ref{thm_uq2_or_t2} to conclude that $\G\cong C(\T^2)$ or $\G\cong
U_\mu(2)$ with $\mu = -\frac{E_{kri}}{E_{kir}}$. 

Now, we can assume that $p_1=p_2=p_3$ and focus on the characteristic
constants. If all of them are different from 1, then according to Proposition
\ref{prop_none_kac_cases}, $\G$ is isomorphic to one of $A_{k,m}(3)$, $U_q(2)$
or $C(\T^2)$. On the other hand, if exactly two characteristic constants equal
1, then, in view of Lemma \ref{lem_iso} and Proposition \ref{prop_some}, $\G$
must be isomorphic to $U_q(2)$ or $C(\T^2)$. 

Finally, it remains to identify the case when $p_1=p_2=p_3$ and $\co{n}{l}=1$
for any $n\neq l$. Then using Lemma \ref{lem_4_solutions} and Proposition
\ref{prop_su_isom} we conclude that $\G$ can be isomorphic to either
$A_{p,-m,m}(3)$ or to $SU_{p,m}(3)$, which ends the proof. 

\section{Final remarks and open problems}
\label{sec_final}

Since the objective of the paper was only to establish a list of PW-quantum
groups and find out which of them are non-trivial, thus the results included in
the paper do not provide a complete description of the structures of 
$SU_{p,m}(3)$ or $A_{p,k,m}(3)$. Both families require separate studies
(description of representations and corepresentations, for instance), which
will be dealt with in a forthcoming paper. 

Some elements in both families can still be isomorphic. In particular, Theorem
\ref{thm_akm_reps} suggests that, in case of $k+m\not\in 3\N$, there might be a
symmetry of $A_{p,k,m}(3)$ with respect to the change $(k,m) \leftrightarrow
(m,k)$. Similarly, for $m\in \{0,1,2\}$ we get three families of quantum groups
$SU_{p,m}(3)$, each one parametrized by $p$. It is still unclear whether
$SU_{p,m_1}(3)$ is isomorphic to $SU_{p,m_1}(3)$ for $m_1\neq m_2$. 

Our results suggest that only very special quantum groups may come from
Woronowicz construction with an array $E$ satisfying the permutation condition.
The natural question arises whether it is possible to provide a general
description of PW-quantum groups for arbitrary dimension.  

It is also an interesting problem to study what are the quantum groups coming
from Woronowicz construction for an array $E$ without the restriction to the
permutation condition. This problem is open even for $N=2$.

\section*{Acknowledgements}
The author wishes to thank Profesor S.\ L.\ Woronowicz and Profesor M.\
Bo\.{z}ejko for very stimulating disscussions. 

The author was supported by the Polish National Science Center
PostDoctoral Fellowship No. 2012/04/S/ST1/00102 while working on this paper. 

\end{document}